\documentclass[11pt]{article}
\usepackage{graphicx}
\usepackage{amsmath}
\usepackage{booktabs}
\usepackage{mathtools}
\usepackage{amssymb}
\usepackage{enumitem}
\usepackage{subcaption}
\usepackage{float}
\usepackage{listings}
\usepackage{xcolor}
\usepackage[nodisplayskipstretch]{setspace}
\usepackage{tikz}
\usetikzlibrary{decorations.markings}
\usetikzlibrary{arrows.meta}
\usepackage{amsthm}
\usepackage[colorlinks,bookmarksopen,bookmarksnumbered,citecolor=blue, linkcolor=red, urlcolor=blue]{hyperref}
\usepackage{cleveref}

\newtheorem{theorem}{Theorem}[section]
\newtheorem{problem}{Problem}[section]

\newtheorem{lemma}{Lemma}[section]

\newtheorem{claim}{Claim}[section]

\allowdisplaybreaks

\makeatletter
\def\leftharpoonfill@{\arrowfill@\leftharpoonup\relbar\relbar}
\def\rightharpoonfill@{\arrowfill@\relbar\relbar\rightharpoonup}
\newcommand\rbjt{\mathpalette{\overarrow@\rightharpoonfill@}}
\newcommand\lbjt{\mathpalette{\overarrow@\leftharpoonfill@}}
\makeatother

\title{Two-block cycles and chromatic number of Hamiltonian digraphs %Without $C(k,\ell)$
}
\author{Ruilin Zheng\thanks{Email: rlzhengmath@163.com}, Junying Lu\thanks{Email: lujunying\_math@163.com}, Xiaolin Wang\thanks{ Email: xiaolinw@fzu.edu.cn}~, Yaojun Chen\thanks{ Corresponding author. Email: yaojunc@nju.edu.cn}
\\
\small{$^{*,\S}$School of Mathematics, Nanjing University, Nanjing
210093, P.R. CHINA}\\
\small{$^\dagger$School of Mathematics and Statistics, Nanjing University of Science and Technology,}\\
\small{
 Nanjing 210094, P.R. China}\\
\small{$^\ddagger$School of Mathematics and Statistics, Fuzhou University, Fuzhou
350108, P.R. China}
}

\date{ }

\begin{document}

\maketitle

\begin{abstract}
Let $k$ and $\ell$ be positive integers. The family $C(k,\ell)$ consists of all digraphs obtained from two internally vertex-disjoint directed paths of lengths at least $k$ and $\ell$, respectively, and  identifying their initial vertices and their terminal vertices. Addario-Berry, Havet and Thomass\'e (JCT-B, 2007) asked whether, for any positive integers $k$ and $\ell$ with $k+\ell \ge 4$, the chromatic number $\chi(D)$ is at most $k+\ell-1$ for every $C(k,\ell)$-free strongly connected  digraph $D$. Let $D$ be a $C(k,\ell)$-free Hamiltonian digraph.  Kim, Kim, Ma and Park (JGT, 2018) showed that $\chi(D) \le k+\ell$ and the bound is attained when $k+\ell=5$.
In this paper, we prove that $\chi(D) \le k+\ell-1$ for $k+\ell\ge 6$ and this bound is best possible for all $k+\ell\geq 6$, 
which resolves the problem posed by Addario-Berry, Havet and Thomass\'e for Hamiltonian digraphs.\\

\noindent {\it AMS classification:} 05C15, 05C20, 05C38, 05C45\\[1mm]
\noindent {\it Keywords:} Chromatic number; Hamiltonian digraphs; Cycles with two blocks
\end{abstract}

\baselineskip=0.202in

\section{Introduction}
Throughout this paper, all graphs and digraphs are simple, that is, containing no loops or multiple edges, although a pair of opposite arcs is allowed in  digraphs. An \emph{oriented} graph is a digraph obtained from a simple graph by replacing each edge with exactly one of the two possible arcs. The length $|P|$ of an oriented path $P$ is the number of arcs it contains. For a digraph $D$, we denote its arc set by $A(D)$. The underlying graph of a digraph $D$ is the simple graph on $V(D)$ in which two vertices are adjacent whenever they are joined by at least one arc of $D$. The chromatic number $\chi(D)$ of a digraph $D$ is defined as the chromatic number of its underlying graph. A digraph $D$ is called $\chi$-chromatic if $\chi(D)=\chi$.

A classical theorem of Gallai~\cite{Gallai} and Roy~\cite{Roy} states that every $\chi$-chromatic digraph contains a directed path on $\chi$ vertices. This motivates the study of oriented graphs that are unavoidable in $\chi$-chromatic digraphs. More precisely, an oriented graph $H$ is called \emph{$\chi$-universal} if every digraph of chromatic number $\chi$ contains $H$ as a subdigraph. 
A \emph{block} of an oriented path or oriented cycle is a maximal directed subpath in it.
%Then an oriented path or oriented cycle is a union of some blocks. 
By applying a classic result of Erd\H{o}s \cite{Erdos},
Cohen, Havet, Lochet and Nisse \cite{Cohen} proved that for any given integers $\chi,m$, there exists a digraph $D$ such that $\chi(D)\ge \chi$ and  every oriented cycle in $D$ has at least $m$ blocks.
Hence, every connected $\chi$-universal oriented graph must be an oriented tree.
In 1980, Burr \cite{Burr} conjectured that every oriented tree of $n$ vertices is $(2n-2)$-universal, which remains open. One can see more results about it in \cite{ElSahili2004,HaggkvistThomason1991,HavetThomasse2000a,HavetThomasse2000b,KuhnMycroftOsthus2011}.

It is also interesting to investigate special families of oriented trees.
A special family of oriented paths, known as oriented paths with two blocks, has received considerable attention.
El-Sahili \cite{ElSahili} conjectured that every oriented path with two blocks on $n\ge 4$ vertices is $n$-universal. Subsequently, El-Sahili and Kouider \cite{ElSahiliKouider} proved that such an oriented path is $(n+1)$-universal, and Addario-Berry, Havet and Thomass\'e \cite{AddarioBerry} later confirmed this conjecture.

Although there exists a digraph with arbitrarily large chromatic number and number of blocks in all oriented cycles, these digraphs are neither tournaments nor strongly connected \cite{Cohen}. So it is also interesting to investigate the existence of oriented cycles in tournaments of order $\chi$ or $\chi$-chromatic strongly connected digraphs.

For tournaments, 
Benhocine and Wojda \cite{Benhocine} proved the existence of $D(n,p)$  in every tournament of order $n$, where  $D(n,p)$ is an oriented cycle with two blocks of order $n$ in which one block has length $p$ for  $1\leq p\leq n-1$. 

\begin{theorem}[Benhocine and Wojda \cite{Benhocine}]
\label{thm:benhocine}
Every tournament $T$ of order $n\ge5$ contains $D(n,p)$ for every $1\le p\le n-1$, unless one of the following holds.
\begin{enumerate}[label=(\arabic*)]
    \item $n=5$, $p=1$, and $T=T_5^*$.
    \item $n=6$, $p=2$, and $T=T_6^*$.
\end{enumerate}
Here $T_5^*$ and $T_6^*$ are  tournaments shown in Figure~\ref{fig:two}.
\end{theorem}

\begin{figure}[htbp]
\centering
\begin{subfigure}[b]{0.48\textwidth}
    \centering
    \begin{tikzpicture}[
        vertex/.style={circle, draw=black, fill=white, line width=0.8pt, minimum size=5pt, inner sep=0pt},
        arr/.style={-{Stealth[length=2.2mm,width=1.5mm]}, draw=black, line width=0.8pt, shorten >=3pt, shorten <=3pt},
        lab/.style={font=\large},
        scale=1.35 
    ]
    \coordinate (T5)  at (90:1.25);
    \coordinate (R5)  at (18:1.25);
    \coordinate (BR5) at (-54:1.25);
    \coordinate (BL5) at (-126:1.25);
    \coordinate (L5)  at (162:1.25);

    \foreach \p/\q in {L5/T5,T5/R5,R5/BR5,BR5/BL5,BL5/L5,L5/R5,BR5/T5,T5/BL5,R5/BL5,L5/BR5}{
        \draw[arr] (\p) -- (\q);
    }

    \foreach \x in {T5,R5,BR5,BL5,L5}{\node[vertex] at (\x) {};}
    \node[lab] at (0,-1.3) {\(T_{5}^{*}\)}; 
    \end{tikzpicture}
    \label{fig:T5}
\end{subfigure}
\hfill
\begin{subfigure}[b]{0.48\textwidth}
    \centering
    \begin{tikzpicture}[
        vertex/.style={circle, draw=black, fill=white, line width=0.8pt, minimum size=5pt, inner sep=0pt},
        arr/.style={-{Stealth[length=2.2mm,width=1.5mm]}, draw=black, line width=0.8pt, shorten >=3pt, shorten <=3pt},
        redarr/.style={arr, draw=red},
        lab/.style={font=\large},
        scale=1.35
    ]
    \coordinate (R6)  at (0:1);
    \coordinate (TR6) at (45:1.414);
    \coordinate (TL6) at (135:1.414);
    \coordinate (L6)  at (180:1);
    \coordinate (BL6) at (225:1.414);
    \coordinate (BR6) at (315:1.414);
    \foreach \p/\q in {TL6/TR6,TR6/R6,R6/BR6,L6/BL6,BL6/BR6,TL6/L6,L6/R6,TL6/BR6,BL6/TR6,L6/TR6,TL6/R6,L6/BR6,BL6/R6}{
        \draw[arr] (\p) -- (\q);
    }
    \foreach \p/\q in {BL6/TL6}{
        \draw[arr] (\p) to[bend left=40] (\q);
    }
    \foreach \p/\q in {BR6/TR6}{
        \draw[arr] (\p) to[bend right=40] (\q);
    }
    \foreach \x in {R6,TR6,TL6,L6,BL6,BR6}{\node[vertex] at (\x) {};}
    \node[lab] at (0,-1.3) {\(T_{6}^{*}\)}; 
    \end{tikzpicture}
    \label{fig:T6}
\end{subfigure}
\caption{The  tournaments $T_{5}^{*}$ and $T_{6}^{*}$.}
\label{fig:two}
\end{figure}
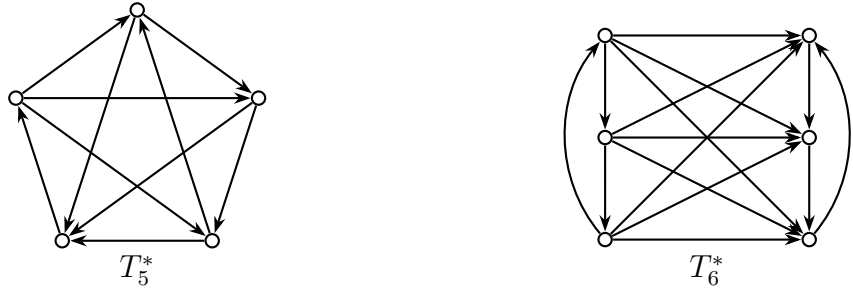

Let $D$ be an $\chi$-chromatic strongly connected digraph.
Bondy \cite{Bondy} proved that  $D$ contains a directed cycle of length at least $\chi$. A natural generalization of Bondy's theorem concerns oriented cycles with two blocks.

The family $C(k,\ell)$ consists of all digraphs obtained from two internally vertex-disjoint directed paths of lengths at least $k$ and $\ell$, respectively, and  identifying their initial vertices and their terminal vertices.
 We say  a digraph is \emph{$C(k,\ell)$-free} if it contains no member of $C(k,\ell)$.
 
 In 2007, Addario-Berry, Havet and Thomass\'e \cite{AddarioBerry} posed the following problem.

\begin{problem}[Addario-Berry, Havet and Thomass\'e \cite{AddarioBerry}]\label{prob-AHT}
Let $k,\ell$ be two positive integers, and let $D$ be an $(k+\ell)$-chromatic strongly connected  digraph with $k+\ell \ge4$. Does $D$ necessarily contain a member of $C(k,\ell)$?
\end{problem}

Cohen, Havet, Lochet and Nisse \cite{Cohen} made the first advance on Problem \ref{prob-AHT}. 

\begin{theorem}[Cohen, Havet, Lochet and Nisse \cite{Cohen}]
Let $k$ and $\ell$ be integers with $k\ge\ell\ge2$ and $k\ge4$, and let $D$ be a  $C(k,\ell)$-free strongly connected digraph. Then
\[
\chi(D)\le(k+\ell-2)(k+\ell-3)(2\ell+2)(k+\ell+1)=O((k+\ell)^4).
\]
\end{theorem}

Subsequently, Kim, Kim, Ma and Park \cite{Kim}  improved this bound by using a different approach.

\begin{theorem}[Kim, Kim, Ma and Park \cite{Kim}]
Let $k$ and $\ell$ be integers with $k\ge\ell\ge1$ and $k\ge2$, and let $D$ be a $C(k,\ell)$-free strongly connected digraph. Then
\[
\chi(D)\le2(2k-3)(k+2\ell-1)<12k^2=O((k+\ell)^2).
\]
\end{theorem}

A digraph $D$ is called \emph{Hamiltonian} if it contains a directed Hamiltonian cycle. The Hamiltonian case is of particular interest because Hamiltonian digraphs form a broad natural subclass of strongly connected digraphs, and a directed Hamiltonian cycle provides a fixed cyclic ordering that facilitates the analysis of forbidden configurations. In this setting, Kim, Kim, Ma and Park \cite{Kim} obtained the following result.

\begin{theorem}[Kim, Kim, Ma and Park \cite{Kim}]
Let $k$ and $\ell$ be positive integers with $k+\ell\ge3$, and let $D$ be a  $C(k,\ell)$-free Hamiltonian digraph. Then its underlying graph is $(k+\ell-1)$-degenerate, and hence $\chi(D)\le k+\ell$.
\end{theorem}

The bound $\chi(D)\le k+\ell$ is best possible when $k+\ell=5$, since the Hamiltonian tournament $T_5^*$ shown in Figure \ref{fig:two} is $C(4,1)$-free and has chromatic number $5$, which also implies that Problem~\ref{prob-AHT} has a negative answer when $k+\ell=5$.
In this paper, we prove that Problem~\ref{prob-AHT} has a positive answer when $D$ is Hamiltonian and $k+\ell\geq 6$. 

\begin{theorem}\label{thm:main-chi-bound}
Let $k\ge\ell\ge1$ be integers with $k+\ell\ge6$,
and let $D$ be a  $C(k,\ell)$-free Hamiltonian digraph. Then the underlying graph of $D$ is $(k+\ell-2)$-degenerate, and hence $\chi(D)\le k+\ell-1$.
\end{theorem}

For all $k+\ell\geq 6$, 
since each member of $C(k,\ell)$ has at least  $k+\ell$ vertices, any Hamiltonian tournament on $k+\ell-1$ vertices shows that the bounds in Theorem~\ref{thm:main-chi-bound} are  tight. 

Theorem \ref{thm:main-chi-bound} is mainly based on the following Theorem \ref{thm:mindeg}.

\begin{theorem}\label{thm:mindeg}
Let $k\ge\ell\ge1$ be integers with $k+\ell\ge6$. Let $D$ be a Hamiltonian digraph, and $G$ be its underlying graph. If $\delta(G)\ge k+\ell-1$, then $D$ contains a member of $C(k,\ell)$.
\end{theorem}

In the end of this section, we introduce some notation. Throughout the rest of this paper, we always let $D$ be a digraph of order $n$, having a Hamiltonian cycle $C=v_0v_1\ldots v_{n-1}v_0$ in which $(v_0,v_1)\in A(D)$. For $0\le i,j\le n-1$,
let $v_i\rbjt{C}v_j$ be the directed path
from $v_i$ to $v_j$ in $C$ and  
$\theta_C(v_i,v_j)=\min\{|v_i\rbjt{C}v_j|,|v_j\rbjt{C}v_i|\}$. Let $G$ be the underlying graph of $D$.
If $v_iv_j\in E(G)$ with $\theta_C(v_i,v_j)=t\geq 2$, then it is called a \emph{$t$-chord}.
An arc $(v_i,v_j)\in A(D)\setminus A(C)$ with $|v_i\rbjt{C}v_j|=t$
is called a \emph{$t$-dichord}. 
If $(v_i,v_j)\in A(D)\setminus A(C)$  is an 
\emph{$(n-t)$-dichord}, then we also call it an \emph{anti-$t$-dichord}. 

For a vertex $v\in V(G)$, its neighborhood is
$N_G(v)=\{u\in V(G)\mid uv\in E(G)\}$, and its closed neighborhood is
$N_G[v]=N_G(v)\cup\{v\}$. Let
$d_G(v)=|N_G(v)|$. When there is no danger of confusion,  we  simply write as $N(v)$, $N[v]$, and $d(v)$. For $D$,
the neighborhood and degree of a vertex are understood to be those in $G$. If the
terminal vertex of a directed path $P$ is the initial vertex of a directed
path $Q$, we write $P\odot Q$ for their concatenation.

The rest of this paper is organized as follows. Section \ref{sec-2} contains preliminary lemmas used throughout the later  proofs. Sections \ref{sec-l1} and \ref{sec-l2} handle the cases $\ell=1$ and $\ell=2$, respectively. In Section \ref{sec-l3}, we discuss $\ell\ge3$, separating into the unbalanced case $k\ge\ell+1$ and the balanced case $k=\ell$. Finally, we prove Theorems \ref{thm:main-chi-bound} and \ref{thm:mindeg} in Section \ref{sec-main}.

\section{Some useful lemmas}\label{sec-2}
In this section, we present several lemmas, which will be used repeatedly in the later  proof. The first three lemmas are easy to prove by the definition of $C(k,\ell)$.
\begin{lemma}\label{lem:rectangle}
Suppose that $D$ contains two vertex-disjoint paths $P_1$ and $P_2$ such that $P_1$ is a directed $(x,y)$-path of length at least $k$ and $P_2$ is a directed $(a,b)$-path of length at least $\ell$.  
If $xa,yb\in E(G)$, then $D$ contains a member of $C(k,\ell)$.
\end{lemma}

\begin{lemma}\label{lem:common-start}
 Suppose that $D$ contains two internally disjoint directed paths $P_x$ and $P_y$   with the same initial vertex $s$  and  distinct terminal vertices $x\in V(P_x)\backslash V(P_y)$ and $y\in V(P_y)\backslash V(P_x)$. Suppose  $|P_x|\ge k$ and $|P_y|\ge \ell$. If $xy\in E(G)$, then $D$ contains a member of $C(k,\ell)$.
\end{lemma}

\begin{lemma}\label{lem:common-end}
Suppose that $D$ contains two internally disjoint directed paths $P_x$ and $P_y$ with distinct initial vertices $x\in V(P_x)\backslash V(P_y)$ and $y\in V(P_y)\backslash V(P_x)$ and the same terminal vertex $t$.  Suppose  $|P_x|\ge k$ and $|P_y|\ge \ell$. If $xy\in E(G)$, then $D$ contains a member of $C(k,\ell)$.   
\end{lemma}

\begin{lemma}[Kim, Kim, Ma and Park \cite{Kim}]\label{lem:boundary-crossing}  
Suppose that \( u, v, x, y \) are four distinct vertices such that \( uv, xy \in E(G) \setminus E(C) \), \( x \in V(u\rbjt{C}v) \)  
and \( y \in V(v\rbjt{C}u) \). For positive integers \( k \) and \( \ell \), if \( |u\rbjt{C}x| \geq k - 1 \) and \( |v\rbjt{C}y| \geq \ell - 1 \),  
then \( D \) contains a member of  \( C(k, \ell) \), unless one of the following occurs:

(a) \( |u\rbjt{C}x| = k - 1 \) and \( (u, v), (y, x) \in A(D) \).

(b) \( |v\rbjt{C}y| = \ell - 1 \) and \( (v, u), (x, y) \in A(D) \).
\end{lemma}

\begin{lemma}\label{lem:no-b-plus-one-dichord}
Let $a\ge3$ and $b\ge2$ be integers. Suppose $n\ge a+b+1$. If $\delta(G)\ge a+b-1$ and $D$ is $C(a,b)$-free, then $C$ contains no $(b+1)$-dichord.
\end{lemma}

\begin{proof}
Without loss of generality, suppose to the contrary that  $(v_0,v_{b+1})\in A(D)$.
We first determine $N[v_1]$. 
Since $n\ge a+b+1$,  $n-a+1>b+1$. For $b+1<i\le n-a+1$, the paths $v_i\rbjt{C}v_{0}\odot(v_0,v_{b+1})$ and $v_1\rbjt{C}v_{b+1}$ have lengths at least $a$ and $b$, respectively. By Lemma~\ref{lem:common-end},  $v_iv_1\notin E(G)$. Hence $N[v_1]\subseteq V(v_{n-a+2}\rbjt{C}v_{b+1})$, where $|V(v_{n-a+2}\rbjt{C}v_{b+1})|=a+b$. Since $|N[v_1]|\ge a+b$, it follows that $N[v_1]=V(v_{n-a+2}\rbjt{C}v_{b+1})$.

Next, we determine  $N[v_b]$. For $a+b\le j\le n-1$, the paths $(v_0,v_{b+1})\odot v_{b+1}\rbjt{C}v_j$ and $v_0\rbjt{C}v_b$ have lengths at least $a$ and $b$, respectively. By Lemma~\ref{lem:common-start}, we have $v_jv_b\notin E(G)$. Hence $N[v_b]\subseteq V(v_0\rbjt{C}v_{a+b-1})$, where $|V(v_0\rbjt{C}v_{a+b-1})|=a+b$. Since $|N[v_b]|\ge a+b$, we have
$N[v_b]=V(v_0\rbjt{C}v_{a+b-1})$.
In particular, $v_bv_{a+b-1}\in E(G)$. If $(v_{a+b-1},v_b)\in A(D)$, then $(v_0,v_{b+1})\odot v_{b+1}\rbjt{C}v_{a+b-1}\odot(v_{a+b-1},v_b)$ and $v_0\rbjt{C}v_b$ form a member of $C(a,b)$, a contradiction. Therefore, $(v_b,v_{a+b-1})\in A(D)$.

Now the paths $v_1\rbjt{C}v_b \odot (v_b,v_{a+b-1})$ and $(v_{n-1},v_0) \odot (v_0,v_{b+1}) \odot v_{b+1}\rbjt{C}v_{a+b-1}$ have lengths $b$ and $a$, respectively. By Lemma~\ref{lem:common-end},  $v_{n-1}v_1\notin E(G)$. Since $a\ge3$,  $v_{n-1}\in V(v_{n-a+2}\rbjt{C}v_{b+1})=N[v_1]$, a contradiction.
\end{proof}

\section{\texorpdfstring{Minimum Degree and $C(k,1)$}
{Minimum Degree and C(k,1)}}
\label{sec-l1}

\begin{theorem}\label{thm:main-l1}
If $n\geq k+2$ and $\delta(G) \geq k\geq 5$, then $D$ contains a member of
$C(k,1)$.

\end{theorem}

\begin{proof}  
Assume that  $n=k+s$  with $s\ge 2$. 
Suppose to the contrary that $D$ is $C(k,1)$-free. 
Since $\delta(G)\ge 5$, $C$ has a chord. 
Choose a dichord $(v_i,v_j)$ such that $|v_i\rbjt{C}v_j|$  is as small as possible. By symmetry, assume $(v_0,v_{t})$ is such a $t$-dichord. Since $D$ is $C(k,1)$-free, we have $2\le t\le k-1$. 

Suppose $t>s$. For any $v_i\in N(v_0)\setminus \{v_1,v_{n-1}\}$, by the choice of $(v_0,v_t)$, if $(v_0,v_i)\in A(D)$, then $t\le i\le k-1$, and if $(v_i,v_0)\in A(D)$, then $t\le n-i\le k-1$, that is, $n-k+1 \le i \le n-t$.
 Since $n=k+s$ and $t>s$, $s+1\leq i\le k+s-t\leq k-1$.
 Combining these two cases, $N(v_0)\subseteq
\{v_1,v_{n-1}\}\cup \{v_{s+1},\ldots,v_{k-1}\}$,
and so $d(v_0)\le 2+(k-s-1)=k-s+1<k$,
a contradiction. Hence, $t\le s$.

Suppose $t\ge 3$. We claim that $N(v_1)\subseteq \{v_0,v_2\}\cup \{v_{t+1},\ldots,v_{k+t-3}\}$.
Suppose $v_i\in N(v_1)$. If $3\le i\le t$, then by the choice of $(v_0,v_t)$,  $(v_i,v_1)\in A(D)$. Since $i\leq t\leq s$, 
$v_i\rbjt{C}v_1\cup (v_i,v_1)$ is a member of $C(k,1)$. 
If $k+t-2\le i\le n-1$, then either $v_1\rbjt{C}v_i\cup (v_1,v_i)$ or $(v_0,v_{t})\odot v_{t}\rbjt{C}v_i\odot (v_i,v_1)\cup (v_0,v_1)$ is a member of $C(k,1)$. Therefore, $N(v_1)\subseteq \{v_0,v_2\}\cup \{v_{t+1},\ldots,v_{k+t-3}\}$ and $d(v_1)\le 2+(k-3)=k-1$, a contradiction.

Let $t=2$, that is, $(v_0,v_2)\in A(D)$. We show that $N(v_1)\subseteq \{v_0,v_2\}\cup \{v_4,\ldots,v_{k}\}$.
If $v_3\in N(v_1)$, then either $v_3\rbjt{C}v_1\cup (v_3,v_1)$  or $(v_1,v_3)\odot v_3\rbjt{C}v_0\odot(v_0,v_2)\cup (v_1,v_2)$ is a member of $C(k,1)$.
If $k+1\le i\le n-1$, since $|(v_0,v_2)\odot v_2\rbjt{C}v_i|\ge k$ and $(v_0,v_1)\in A(D)$, by Lemma \ref{lem:common-start},  $v_i\notin N(v_1)$. Therefore, $d(v_1)\le 2+(k-3)=k-1$, a contradiction.
\end{proof}

\section{\texorpdfstring{Minimum Degree and $C(k,2)$}
{Minimum Degree and C(k,2)}}\label{sec-l2}

\begin{theorem}\label{thm:L2} Suppose $n\ge k+3\ge 7$.
If $\delta(G) \geq k+1$, then $D$  contains a member of $C(k,2)$. 
\end{theorem} 
\begin{proof}  
Suppose to the contrary  that $D$ is $C(k,2)$-free. 
We first prove the following claim.

\begin{claim}\label{lem:L24}
For every $s\geq 0$, $C$ contains no $(3+s)$-dichord when $n\geq k+3+s$ and no anti-$(3+s)$-dichord when $n\geq k+4+s$.
\end{claim}

\begin{proof} We use induction on $s$. Suppose  $s=0$. 
 The assertion for $3$-dichords follows immediately from Lemma~\ref{lem:no-b-plus-one-dichord} with $(a,b)=(k,2)$. Now we  suppose  $n\ge k+4$, and by cyclic symmetry,  suppose  $(v_3,v_0)\in A(D)$.

We first determine $N[v_t]$ for $t=1,2$.
Since $n \geq k+4$, for $k+3\leq i \leq n-1$, note that $v_3\rbjt{C}v_i$ and $(v_3,v_0) \odot v_0\rbjt{C}v_t$ have lengths at least $k$ and $2$, by Lemma \ref{lem:common-start}, we have $v_tv_i \notin E(G)$. Thus, $N[v_t] \subseteq V(v_0\rbjt{C}v_{k+2})$.
Since $n-k>3$, for $3<j \leq n-k$, note that $v_j\rbjt{C}v_0$ and $v_t\rbjt{C}v_3 \odot (v_3,v_0)$ have lengths at least $k$ and $2$, respectively, by Lemma \ref{lem:common-end}, we have $v_tv_j \notin E(G)$, and thus $N[v_t] \subseteq V(v_{n-k+1}\rbjt{C}v_{3})$.
Therefore,  $N[v_t] \subseteq V(v_{n-k+1}\rbjt{C}v_{3}) \cap V(v_0\rbjt{C}v_{k+2})$. 

If $k+2 \leq n-k+1$, then  $N[v_t]  \subseteq  V(v_0\rbjt{C}v_3)\cup \{v_{k+2}\}$, which contradicts $|N[v_t]|\ge 6$.
If $k+2 > n-k+1$, then  
$N[v_t]  \subseteq  V(v_0\rbjt{C}v_3) \cup V(v_{n-k+1}\rbjt{C}v_{k+2})$, and 
$|V(v_0\rbjt{C}v_3) \cup V(v_{n-k+1}\rbjt{C}v_{k+2})|=4+(k+2)-(n-k+1)+1=2k-n+6$.
Thus, if 
$n \geq k+5$, then we have
$2k-n+6 \leq k+1 <|N[v_t]|$, a contradiction. Hence $n=k+4$ and $2k-n+6=k+2$.
Since $|N[v_t]|\ge k+2$, we can deduce
$N[v_t]=V(v_0\rbjt{C}v_3) \cup V(v_{n-k+1}\rbjt{C}v_{k+2})$ for $t=1,2$.

In particular, $v_{n-k+1}v_2 \in E(G)$. As shown in Figure \ref{fig:1}:
If $(v_2,v_{n-k+1}) \in A(D)$, then $(v_2,v_{n-k+1}) \odot  v_{n-k+1}\rbjt{C}v_{0} $ and $(v_2,v_3) \odot (v_3,v_0)$ form a member of $C(k,2)$, a contradiction.
If $(v_{n-k+1},$ $v_2) \in A(D)$, then $v_{n-k+1}\rbjt{C}v_1$ and $(v_{n-k+1},v_2) \odot (v_2,v_3)$ have lengths $k$ and $2$, respectively. By Lemma \ref{lem:common-start}, we have $v_{1}v_3 \notin E(G)$, which contradicts $v_3\in N[v_1]$. 

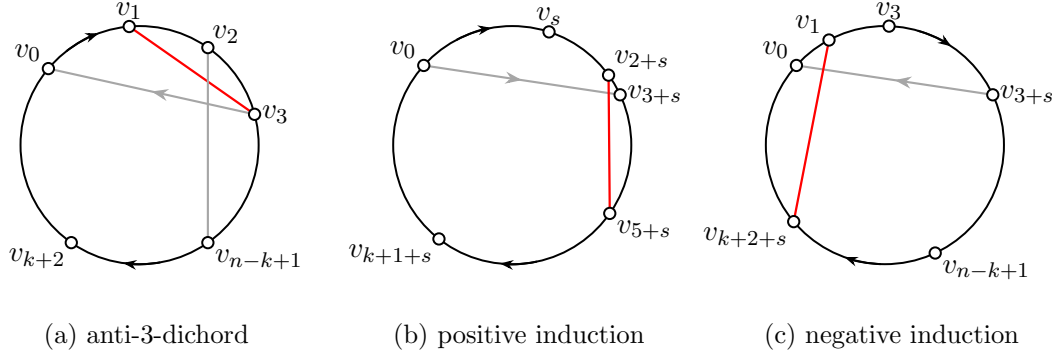
\begin{figure}[htbp]
\centering

% ---------- anti-3-dichord ----------
\begin{subfigure}[b]{0.32\textwidth}
\centering
\begin{tikzpicture}[
    scale=0.9,
    every node/.style={
        font=\normalfont\normalsize,
        inner sep=0pt
    },
    vertex/.style={
        circle,
        draw=black,
        fill=white,
        line width=0.75pt,
        inner sep=1.5pt
    },
    edge/.style={line width=0.75pt},
    rededge/.style={line width=0.85pt, red},
    grayedge/.style={line width=0.8pt, gray!70},
    cwmark/.style={
        line width=0.75pt,
        -{Stealth[length=1.8mm,width=1.3mm]}
    },
    midarrow/.style={
        postaction={decorate},
        decoration={
            markings,
            mark=at position 0.5 with {
                \arrow{Stealth[length=2.2mm,width=1.6mm]}
            }
        }
    }
]

% fixed bounding box
\path[use as bounding box]
    (-2.42,-2.35) rectangle (2.62,2.35);

\def\R{1.75}

% vertices
\coordinate (v1) at (95:\R);
\coordinate (v2) at (55:\R);
\coordinate (v3) at (15:\R);
\coordinate (vnk1) at (-55:\R);
\coordinate (vk2) at (-125:\R);
\coordinate (v0) at (140:\R);

% circle
\draw[edge] (0,0) circle (\R);

% chords
\draw[rededge] (v1) -- (v3);
\draw[grayedge] (v2) -- (vnk1);

% midpoint arrow from v3 to v0
\draw[grayedge,midarrow] (v3) -- (v0);

% vertices
\node[vertex] at (v1) {};
\node[vertex] at (v2) {};
\node[vertex] at (v3) {};
\node[vertex] at (vnk1) {};
\node[vertex] at (vk2) {};
\node[vertex] at (v0) {};

% clockwise direction marks, drawn on top
\draw[cwmark]
    (128:\R)
    arc[start angle=128, end angle=110, radius=\R];

\draw[cwmark]
    (-80:\R)
    arc[start angle=-80, end angle=-98, radius=\R];

% labels
\node[above=2.6pt] at (v1) {$v_1$};
\node[above right=1.6pt] at (v2) {$v_2$};
\node[right=2.6pt] at (v3) {$v_3$};
\node[below right=2.5pt] at (vnk1) {$v_{n-k+1}$};
\node[below left=2pt] at (vk2) {$v_{k+2}$};
\node[above left=2.5pt] at (v0) {$v_0$};

\end{tikzpicture}
\caption{anti-3-dichord}
\label{fig:1}
\end{subfigure}
\hfill
% ---------- positive induction ----------
\begin{subfigure}[b]{0.32\textwidth}
\centering
\begin{tikzpicture}[
    scale=0.9,
    every node/.style={
        font=\normalfont\normalsize,
        inner sep=0pt
    },
    vertex/.style={
        circle,
        draw=black,
        fill=white,
        line width=0.75pt,
        inner sep=1.5pt
    },
    edge/.style={line width=0.75pt},
    rededge/.style={line width=0.85pt, red},
    grayedge/.style={line width=0.8pt, gray!70},
    cwmark/.style={
        line width=0.75pt,
        -{Stealth[length=1.8mm,width=1.3mm]}
    },
    midarrow/.style={
        postaction={decorate},
        decoration={
            markings,
            mark=at position 0.5 with {
                \arrow{Stealth[length=2.2mm,width=1.6mm]}
            }
        }
    }
]

% fixed bounding box
\path[use as bounding box]
    (-2.42,-2.35) rectangle (2.62,2.35);

\def\R{1.75}

% vertices
\coordinate (v0) at (138:\R);
\coordinate (vs) at (72:\R);
\coordinate (v2s) at (36:\R);
\coordinate (v3s) at (25:\R);
\coordinate (v5s) at (-35:\R);
\coordinate (vk1s) at (-128:\R);

% circle
\draw[edge] (0,0) circle (\R);

% chords
% midpoint arrow from v0 to v_{3+s}
\draw[grayedge,midarrow] (v0) -- (v3s);

\draw[rededge] (v2s) -- (v5s);

% vertices
\node[vertex] at (v0) {};
\node[vertex] at (vs) {};
\node[vertex] at (v2s) {};
\node[vertex] at (v3s) {};
\node[vertex] at (v5s) {};
\node[vertex] at (vk1s) {};

% clockwise direction marks, drawn on top
\draw[cwmark]
    (118:\R)
    arc[start angle=118, end angle=100, radius=\R];

\draw[cwmark]
    (-78:\R)
    arc[start angle=-78, end angle=-96, radius=\R];

% labels
\node[above left=2.5pt] at (v0) {$v_0$};
\node[above=2.5pt] at (vs) {$v_s$};
\node[above right=2pt] at (v2s) {$v_{2+s}$};
\node[right=2.6pt] at (v3s) {$v_{3+s}$};
\node[below right=2.5pt] at (v5s) {$v_{5+s}$};
\node[below left=2pt] at (vk1s) {$v_{k+1+s}$};

\end{tikzpicture}
\caption{positive induction}
\label{fig:positive}
\end{subfigure}
\hfill
% ---------- negative induction ----------
\begin{subfigure}[b]{0.32\textwidth}
\centering
\begin{tikzpicture}[
    scale=0.9,
    every node/.style={
        font=\normalfont\normalsize,
        inner sep=0pt
    },
   vertex/.style={
        circle,
        draw=black,
        fill=white,
        line width=0.75pt,
        inner sep=1.5pt
    },
    edge/.style={line width=0.75pt},
    rededge/.style={line width=0.85pt, red},
    grayedge/.style={line width=0.8pt, gray!70},
    cwmark/.style={
        line width=0.75pt,
        -{Stealth[length=1.8mm,width=1.3mm]}
    },
    midarrow/.style={
        postaction={decorate},
        decoration={
            markings,
            mark=at position 0.5 with {
                \arrow{Stealth[length=2.2mm,width=1.6mm]}
            }
        }
    }
]

% fixed bounding box
\path[use as bounding box]
    (-2.42,-2.35) rectangle (2.62,2.35);

\def\R{1.75}

% vertices
\coordinate (v0) at (138:\R);
\coordinate (v1) at (118:\R);
\coordinate (v3) at (88:\R);
\coordinate (v3s) at (25:\R);
\coordinate (vnk1) at (-65:\R);
\coordinate (vk2s) at (-140:\R);

% circle
\draw[edge] (0,0) circle (\R);

% chords
% midpoint arrow from v_{3+s} to v0
\draw[grayedge,midarrow] (v3s) -- (v0);

\draw[rededge] (v1) -- (vk2s);

% vertices
\node[vertex] at (v0) {};
\node[vertex] at (v1) {};
\node[vertex] at (v3) {};
\node[vertex] at (v3s) {};
\node[vertex] at (vnk1) {};
\node[vertex] at (vk2s) {};

% clockwise direction marks, drawn on top
\draw[cwmark]
    (74:\R)
    arc[start angle=74, end angle=56, radius=\R];

\draw[cwmark]
    (-92:\R)
    arc[start angle=-92, end angle=-110, radius=\R];

% labels
\node[above left=2.5pt] at (v0) {$v_0$};
\node[above left=2.5pt] at (v1) {$v_1$};
\node[above=2.5pt] at (v3) {$v_3$};
\node[right=2.5pt] at (v3s) {$v_{3+s}$};
\node[below right=2.5pt] at (vnk1) {$v_{n-k+1}$};
\node[below left=2pt] at (vk2s) {$v_{k+2+s}$};

\end{tikzpicture}
\caption{negative induction}
\label{fig:negative}
\end{subfigure}

\caption{Illustrations for Claim~\ref{lem:L24}}
\label{fig:induction}
\end{figure}
 
Assume that $s\ge 1$ and the result holds for all smaller values $s$.

Suppose that $n\geq k+3+s$ and $(v_0,v_{3+s})\in A(D)$ as shown in Figure~\ref{fig:positive}. Consider $N[v_{2+s}]$.
For $k+2+s\le j \le n-1$, note that $(v_0,v_{3+s})\odot v_{3+s}\rbjt{C}v_j$ and $v_0\rbjt{C}v_{2+s}$ have lengths at least $ k$ and $2$, respectively, by Lemma \ref{lem:common-start}, we have $v_jv_{2+s}\notin E(G)$.
Hence $N[v_{2+s}]\subseteq V(v_0\rbjt{C}v_{k+1+s})$.
By induction hypothesis, for any $0\le i\le s-1$, $v_iv_{2+s} \notin E(G)$. Moreover, note that $v_{5+s}\in V(v_0\rbjt{C}v_{k+1+s})$ and $v_{2+s}v_{5+s}$ is a $3$-chord, by the induction basis, $v_{2+s}v_{5+s} \notin E(G)$.
Therefore, $N[v_{2+s}]\subseteq V(v_s\rbjt{C}v_{k+1+s})\setminus\{v_{5+s}\}$,
contradicting $|N[v_{2+s}]|\ge k+2$.

Suppose $n\ge k+4+s$ and $(v_{3+s},v_0)\in A(D)$ as shown in Figure~\ref{fig:negative}.
Consider $N[v_1]$.  For $k+3+s \le i\le n-1$, note that $v_{3+s}\rbjt{C}v_i$ and $(v_{3+s},v_0)\odot (v_0,v_1)$ have lengths at least $k$ and $2$, respectively, by Lemma \ref{lem:common-start}, $v_1v_i\notin E(G)$.
Thus $N[v_1]\subseteq V(v_0\rbjt{C}v_{k+2+s})$.
Since $n\ge k+4+s$, we have $n-k>3+s$.  For $3+s<j\le n-k$, note that $v_j\rbjt{C}v_0$ and $v_1\rbjt{C}v_{3+s}\odot(v_{3+s},v_0)$ have lengths  at least $k$ and $2$, respectively, by Lemma \ref{lem:common-end}, $v_1v_j\notin E(G)$. By induction hypothesis, for any $v_i$ with $4\le i\le 3+s$, $v_1v_i \notin E(G)$.
Thus $N[v_1]\subseteq V(v_{n-k+1}\rbjt{C}v_{3})$. Consequently $N[v_1]\subseteq V(v_{n-k+1}\rbjt{C}v_{3})\cap V(v_0\rbjt{C}v_{k+2+s})$. 

If $k+2+s\le n-k+1$, then $N[v_1]\subseteq V(v_0\rbjt{C}v_3) \cup \{v_{k+2+s}\}$, contradicting $|N[v_1]|\ge k+2 \ge 6$.
If $k+2+s>n-k+1$, then $N[v_1]\subseteq V(v_0\rbjt{C}v_3) \cup V(v_{n-k+1}\rbjt{C}v_{k+2+s})$, and
$|V(v_0\rbjt{C}v_3)\cup V(v_{n-k+1}\rbjt{C}v_{k+2+s})| = 4+(k+2+s)-(n-k+1)+1 = 2k-n+6+s$.
Since $n\ge k+4+s$, $2k-n+6+s \le k+2$, with equality only if $n=k+4+s$. Thus note that $|N[v_1]|\ge k+2$, we have $n=k+4+s$ and $N[v_1]=V(v_0\rbjt{C}v_3)\cup V(v_{n-k+1}\rbjt{C}v_{k+2+s})$. In particular, $v_{k+2+s}v_{1}\in E(G)$.  Since $n=k+4+s$, we have $k+2+s=n-2$, and so $v_{n-2}v_{1}$ is a $3$-chord, contradicting the induction basis.
\end{proof}

If $n\ge k+5$, let $n=k+4+t$ with $t\ge1$.
By Claim \ref{lem:L24}, $v_0$ is nonadjacent to $v_3,v_4,\ldots,v_{t+3},v_{n-4}, v_{n-3}$.
These $t+3$ vertices are distinct since $n-4=k+t>t+3$. Consequently $d(v_0)\le(n-1)-(t+3)=k$,
contradicting $\delta(G)\ge k+1$.  

\vskip 1mm

Let $n=k+3$. Since $\delta(G) \ge k+1$, $E(\overline{G})$ is a matching. We first claim that there exist successive four vertices, say $v_0, v_1,v_2,v_3$, such that $v_0v_2,v_1v_3 \in E(G)$. Without loss of generality, suppose $v_0v_2\notin E(G)$. Since $E(\overline{G})$ is a matching, $v_2v_4\in E(G)$. If $v_3v_5\in E(G)$, then $v_2,v_3,v_4,v_5$ are required. If $v_3v_5\notin E(G)$, then $v_1v_3\in E(G)$. And $v_1,v_2,v_3,v_4$ are required. Hence there exist successive four vertices, say $v_0, v_1,v_2,v_3$, such that $v_0v_2,v_1v_3 \in E(G)$. Applying Lemma \ref{lem:boundary-crossing} on $(x,v,y,u)=(v_0,v_1,v_2,v_3)$, we get $(v_0,v_2), (v_1,v_3)\in A(D)$. 
If $(v_{1},v_{k+1}) \in A(D)$, then $(v_0,v_2) \odot v_2\rbjt{C}v_{k+1}$ and $(v_0,v_1) \odot (v_{1},v_{k+1})$ form a member of $C(k,2)$, and if $(v_{k+1},v_{1}) \in A(D)$, then this arc is a 3-dichord, contradicting Claim \ref{lem:L24}. Thus $v_1v_{k+1} \notin E(G)$.
If $(v_{4},v_{1}) \in A(D)$, then $v_4\rbjt{C}v_{0}\odot (v_0,v_2)$ and $(v_4,v_1) \odot (v_{1},v_{2})$ form a member of $C(k,2)$, and if $(v_{1},v_{4})\in A(D)$, then this arc is a 3-dichord, contradicting Claim \ref{lem:L24}. Thus $v_1v_{4} \notin E(G)$. Therefore, $|N[v_{1}]| \leq n-2=k+1$, a contradiction. 

\vskip 1mm

Let $n=k+4$.  By Claim \ref{lem:L24}, there are no $3$-chords. So $N[v_i]=V(D)\setminus\{v_{i-3},v_{i+3}\}$.
In particular, $N[v_0]=V(D)\setminus\{v_3,v_{n-3}\}$ and $N[v_1]=V(D)\setminus\{v_4,v_{n-2}\}$, so $v_0v_4,v_1v_{n-1} \in E(G)$. By Claim \ref{lem:L24}, there are no $4$-dichords, and hence $(v_4,v_0) \in A(D)$.  
If $(v_{n-1},v_1) \in A(D)$, then
$v_4\rbjt{C}v_{n-1}\odot (v_{n-1},v_1)$ and $(v_4,v_0)\odot (v_0,v_1)$
form a member of $C(k,2)$, a contradiction. 
If $(v_{1},v_{n-1}) \in A(D)$, then
$v_2\rbjt{C}v_{n-1}$ and $(v_0,v_1)\odot (v_1,v_{n-1})$ have lengths $k+1$ and $2$, respectively. By Lemma \ref{lem:common-end}, $v_0v_2\notin E(G)$, contradicting $N[v_0]=V(D)\setminus\{v_3,v_{n-3}\}$.
\end{proof}

\section{\texorpdfstring{Minimum Degree and $C(k,\ell)$}
{Minimum Degree and C(k,l)} for \texorpdfstring{$\ell \geq 3$}{l>=3}}\label{sec-l3}

In this section, let $k,\ell$ be two positive integers with $k, \ell\ge 3$. When describing $N[v_i]$ for some $v_i \in V(D)$, the subscripts are taken modulo $n$. We will prove the following two theorems.

\begin{theorem}\label{thm:Ll+1}
Let  $k \geq \ell+1 \geq 4$ and $n\ge k+\ell+1$. If $\delta(G) \geq k+\ell-1$, then $D$ contains a member of $C(k,\ell)$.
\end{theorem}

\begin{theorem}\label{thm:Ll}
Let  $k=\ell \geq 3$ and $n\ge 2\ell+1$. If $\delta(G) \geq 2\ell-1$, then $D$ contains a member of  $C(\ell,\ell)$.
\end{theorem}

Before proving, we give some lemmas.

\begin{lemma}\label{lem:Ll}
Let $k, \ell\ge 3$ and $n\ge k+\ell+2$. If $\delta(G) \geq k+\ell-1$ and $D$ is $C(k,\ell)$-free, then $C$ contains neither  anti-$\ell$-dichords nor anti-$k$-dichords.
\end{lemma}

\begin{proof} By symmetry of $k,\ell$, it suffices to prove that $C$ contains no anti-$\ell$-dichord. 
Without loss of generality, suppose to the contrary that $(v_\ell,v_0)\in A(D)$.
Then $Z=v_0v_1\cdots v_\ell v_0$ is a directed cycle of length $\ell+1$.  
For convenience, set
\[
 W=w_1w_2\cdots w_m:=v_{\ell+1}v_{\ell+2}\cdots v_{n-1},
 \quad m=n-\ell-1\ge k+1. 
\]

For $0\le i\le\ell$, let $X_i=\{x:v_iw_x\in E(G)\}$.
Since $v_i$ has at most $\ell$ neighbors in $Z$ and $\delta(G)\ge k+\ell-1$, we have $|X_i|\ge k-1$.

Since $v_\ell \rbjt{Z}v_{\ell-1}$ has length $\ell$, and $v_\ell \rbjt{C}w_x$ has length at least $k$ for $x\ge k$, by Lemma \ref{lem:common-start}, $v_{\ell-1}w_x\notin E(G)$. Then  $X_{\ell-1}\subseteq \{1,\ldots,k-1\}$ and hence $X_{\ell-1}=\{1,\ldots,k-1\}$.
Similarly, we have $X_1=\{m-k+2,\ldots,m\}$.

Let $\alpha_i=\min X_i$ and $\beta_i=\max X_i$. As shown in  Figure \ref{fig:antil}: Note that $v_i\rbjt{Z}v_{i-1}$ has length $\ell$ for any $1\le i\le \ell$, and $w_x\rbjt{C}w_y$ with $1\le x<y\le m$ has length $y-x$, if $v_iw_x, v_{i-1}w_y\in E(G)$, then by Lemma \ref{lem:rectangle}, $y-x\le k-1$. Thus we have $\beta_{i-1}\le \alpha_i+k-1$. Because $|X_i|\ge k-1$, $\alpha_i\le \beta_i-k+2$. So $\beta_{i-1}\le \beta_i+1$ and $\alpha_{i-1}\le \alpha_i+1$.

We first show $v_{\ell}w_m \in E(G)$. Since $\beta_{i-1}\le \beta_i+1$, we have $m=\beta_1\le \beta_{\ell-1}+\ell-2=k+\ell-3$ and $\beta_{\ell-1+k-m}\le \beta_{\ell-1}+m-k=m-1$, thus $v_{\ell-1+k-m},\ldots,v_{\ell-1}$ is nonadjacent to $w_m$. Then $|N[w_m]| \le n-(m-k+1)=k+\ell$, which implies the vertices nonadjacent to $w_m$ in $G$ are precisely $v_{\ell-1+k-m},\ldots,v_{\ell-1}$, and hence $v_{\ell}w_m \in E(G)$.

Next we prove $v_{0}w_1 \in E(G)$. Since $\alpha_{i-1}\le \alpha_i+1$,  $\alpha_1=m-k+2$ and $m\leq k+\ell-3$, we obtain $\alpha_i\ge m-k+3-i (1\le i\le m-k+1)$.
In particular, $\alpha_i\ge2$ for $1\le i\le m-k+1$, so $w_1$ is nonadjacent to $v_1,v_2,\ldots,v_{m-k+1}$. Then $|N[w_1]|\le n-(m-k+1)=k+\ell$. Since $|N[w_1]|\ge k+\ell$, these vertices are exactly the non-neighbors of $w_1$ in $G$, and hence $v_0w_1\in E(G)$.

Since $w_1\rbjt{C}w_m$ and $v_0\rbjt{C}v_\ell$ have lengths $m-1\ge k$ and $\ell$, respectively, the two edges $v_0w_1$ and $v_\ell w_m$ yield a contradiction by Lemma~\ref{lem:rectangle}. This proves Lemma~\ref{lem:Ll}.
\end{proof}

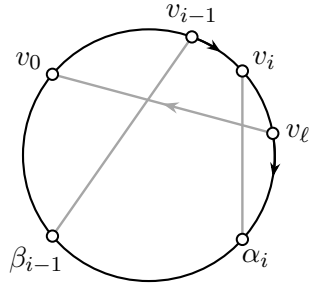
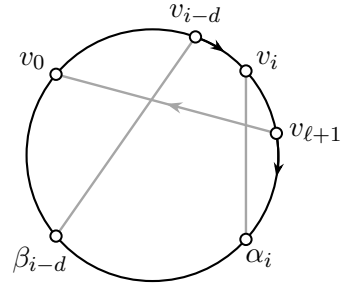
\begin{figure}[htbp]
\centering

% ---------- anti-\ell ----------
\begin{subfigure}[b]{0.45\textwidth}
\centering
\begin{tikzpicture}[
    scale=0.9,
    every node/.style={font=\normalfont\normalsize},
    vertex/.style={
        circle,
        draw=black,
        fill=white,
        line width=0.75pt,
        inner sep=1.5pt
    },
    edge/.style={line width=0.8pt},
    grayedge/.style={line width=0.9pt, gray!70},
    cwmark/.style={
        line width=0.8pt,
        -{Stealth[length=1.9mm,width=1.35mm]}
    },
    midarrow/.style={
        postaction={decorate},
        decoration={
            markings,
            mark=at position 0.5 with {
                \arrow{Stealth[length=2.1mm,width=1.5mm]}
            }
        }
    }
]

\def\R{1.85}

% fixed bounding box for alignment
\path[use as bounding box]
    (-2.45,-2.35) rectangle (2.55,2.35);

% vertices
\coordinate (v0) at (140:\R);
\coordinate (vim) at (70:\R);
\coordinate (vi) at (42:\R);
\coordinate (vell) at (10:\R);
\coordinate (alphai) at (-42:\R);
\coordinate (betaim) at (-140:\R);

% circle and clockwise marks
\draw[edge] (0,0) circle (\R);

\draw[cwmark]
    (72:\R)
    arc[start angle=72, end angle=55, radius=\R];

\draw[cwmark]
    (8:\R)
    arc[start angle=8, end angle=-10, radius=\R];

% chords
% midpoint arrow from v_\ell to v_0
\draw[grayedge,midarrow] (vell) -- (v0);

\draw[grayedge] (vim) -- (betaim);
\draw[grayedge] (vi) -- (alphai);

% vertices
\node[vertex] at (v0) {};
\node[vertex] at (vim) {};
\node[vertex] at (vi) {};
\node[vertex] at (vell) {};
\node[vertex] at (alphai) {};
\node[vertex] at (betaim) {};

% labels
\node[above left=0pt,yshift=-2pt] at (v0) {$v_0$};
\node[above=1pt] at (vim) {$v_{i-1}$};
\node[above right=0pt, yshift=-3pt] at (vi) {$v_i$};
\node[right=1pt] at (vell) {$v_{\ell}$};
\node[below right=0pt,xshift=-4pt] at (alphai) {$\alpha_i$};
\node[below left=0pt, xshift=8pt] at (betaim) {$\beta_{i-1}$};

\end{tikzpicture}
\caption{anti-$\ell$-dichord}
\label{fig:antil}
\end{subfigure}
\hspace{0.01\textwidth}
% ---------- anti-(\ell+1) ----------
\begin{subfigure}[b]{0.45\textwidth}
\centering
\begin{tikzpicture}[
    scale=0.9,
    every node/.style={font=\normalfont\normalsize},
    vertex/.style={
        circle,
        draw=black,
        fill=white,
        line width=0.75pt,
        inner sep=1.5pt
    },
    edge/.style={line width=0.8pt},
    grayedge/.style={line width=0.9pt, gray!70},
    cwmark/.style={
        line width=0.8pt,
        -{Stealth[length=1.9mm,width=1.35mm]}
    },
    midarrow/.style={
        postaction={decorate},
        decoration={
            markings,
            mark=at position 0.5 with {
                \arrow{Stealth[length=2.1mm,width=1.5mm]}
            }
        }
    }
]

\def\R{1.85}

% fixed bounding box for alignment
\path[use as bounding box]
    (-2.45,-2.35) rectangle (2.55,2.35);

% vertices
\coordinate (v0) at (140:\R);
\coordinate (vimd) at (70:\R);
\coordinate (vi) at (42:\R);
\coordinate (vellp) at (10:\R);
\coordinate (alphai) at (-42:\R);
\coordinate (betaimd) at (-140:\R);

% circle and clockwise marks
\draw[edge] (0,0) circle (\R);

\draw[cwmark]
    (72:\R)
    arc[start angle=72, end angle=55, radius=\R];

\draw[cwmark]
    (8:\R)
    arc[start angle=8, end angle=-10, radius=\R];

% chords
% midpoint arrow from v_{\ell+1} to v_0
\draw[grayedge,midarrow] (vellp) -- (v0);

\draw[grayedge] (vimd) -- (betaimd);
\draw[grayedge] (vi) -- (alphai);

% vertices
\node[vertex] at (v0) {};
\node[vertex] at (vimd) {};
\node[vertex] at (vi) {};
\node[vertex] at (vellp) {};
\node[vertex] at (alphai) {};
\node[vertex] at (betaimd) {};

% labels
\node[above left=0pt,yshift=-2pt] at (v0) {$v_0$};
\node[above=1pt] at (vimd) {$v_{i-d}$};
\node[above right=0pt, yshift=-3pt] at (vi) {$v_i$};
\node[right=1pt] at (vellp) {$v_{\ell+1}$};
\node[below right=0pt,xshift=-4pt] at (alphai) {$\alpha_i$};
\node[below left=0pt, xshift=8pt] at (betaimd) {$\beta_{i-d}$};
\end{tikzpicture}
\caption{anti-$(\ell+1)$-dichord}
\label{fig:antil+1}
\end{subfigure}

\caption{Illustrations for Lemmas~\ref{lem:Ll} and \ref{lem:Ll+1}}
\label{fig:2}
\end{figure}

\begin{lemma}\label{lem:Ll+1}
Let  $k, \ell \geq 3 $. If $\delta(G) \geq k+\ell-1$ and $D$ is $C(k,\ell)$-free, then $C$  contains neither $(\ell+1)$-dichords nor  $(k+1)$-dichords when $n \geq k+\ell+1$ and neither  anti-$(\ell+1)$-dichords nor  anti-$(k+1)$-dichords when $n \geq k+\ell+2$.
\end{lemma}

\begin{proof}
 By symmetry of $k$ and $\ell$, we only need to prove that $C$ contains no $(\ell+1)$-dichords when $n \geq k+\ell+1$ and no anti-$(\ell+1)$-dichords when $n \geq k+\ell+2$.
Applying Lemma~\ref{lem:no-b-plus-one-dichord} with $(a,b)=(k,\ell)$,  $C$ contains no $(\ell+1)$-dichords when $n \geq k+\ell+1$. 

Let $n\geq k+\ell+2$.
Without loss of generality, suppose to the contrary that $(v_{\ell+1},v_0)\in A(D)$. Then $Z=v_0v_1\cdots v_{\ell+1}v_0$ is a directed cycle. Set $W=w_1w_2\cdots w_m:=v_{\ell+2}v_{\ell+3}\cdots v_{n-1}$, where $m=n-\ell-2\ge k$.
For $0\le i\le\ell+1$, let $X_i=\{x:v_iw_x\in E(G)\}$. Since $v_i$ has at most $\ell+1$ neighbors in $Z$ and $\delta(G)\ge k+\ell-1$,  $|X_i|\ge k-2$.

Consider $X_{\ell-1}$. Note that $(v_{\ell+1},v_0)\odot v_0\rbjt{C}v_{\ell-1}$ has length $\ell$, and  $v_{\ell+1}\rbjt{C}w_x$ has length at least $k$ for $x\ge k$. By Lemma~\ref{lem:common-start}, $v_{\ell-1}w_x\notin E(G)$, and hence $X_{\ell-1}\subseteq\{1,\ldots,k-1\}$.  If $(w_{k-1},v_{\ell-1})\in A(D)$, then $v_{\ell+1}\rbjt{C}w_{k-1}\odot(w_{k-1},v_{\ell-1})$ and $(v_{\ell+1},v_0)\odot v_0\rbjt{C}v_{\ell-1}$ form a member of $C(k,\ell)$. If $(v_{\ell-1},w_{k-1})\in A(D)$, then it is a $(k+1)$-dichord, a contradiction. Hence $X_{\ell-1}\subseteq\{1,\ldots,k-2\}$, and therefore $X_{\ell-1}=\{1,\ldots,k-2\}$.

Consider $X_{2}$. Note that $v_2\rbjt{C}v_{\ell+1}\odot(v_{\ell+1},v_0)$ has length $\ell$, and $w_x\rbjt{C}v_0$ has length at least $k$ for $x\le m-k+1$. By Lemma \ref{lem:common-end}, $v_2w_x\notin E(G)$. Moreover, if  $(v_2,w_{m-k+2})\in A(D)$, then $(v_2,w_{m-k+2})\odot w_{m-k+2}\rbjt{C}v_0$ and $v_2\rbjt{C}v_{\ell+1}\odot(v_{\ell+1},v_0)$ form a member of $C(k,\ell)$. If $(w_{m-k+2},v_2) \in A(D)$, it is a $(k+1)$-dichord, a contradiction. Hence $X_2\subseteq\{m-k+3,\ldots,m\}$, and therefore $X_2=\{m-k+3,\ldots,m\}$.

If $\ell=3$, then $X_{\ell-1}=X_2$, forcing $\{1,\ldots,k-2\}=\{m-k+3,\ldots,m\}$, which is impossible since $m\ge k$. Hence we may assume $\ell\ge4$.

Let $\alpha_i=\min X_i$ and $\beta_i=\max X_i$. For $d\in\{1,2\}$,  $v_i\rbjt{Z}v_{i-d}$ has length at least $\ell$. Therefore, as shown in Figure \ref{fig:antil+1}, if $1\le x<y\le m$ and $v_iw_x,v_{i-d}w_y\in E(G)$, then Lemma~\ref{lem:rectangle} gives $y-x\le k-1$. Thus $\beta_{i-d}\le \alpha_i+k-1$ for $d=1,2$. Because $|X_i|\ge k-2$,  $\alpha_i\le \beta_i-k+3$, and hence $\beta_{i-d}\le \beta_i+2$ and $\alpha_{i-d}\le \alpha_i+2$ for $d=1,2$.

Since $\beta_{i-d}\le \beta_i+2$,  $m=\beta_2\le \beta_{\ell-1}+2\lceil(\ell-3)/2\rceil=k-2+2\lceil(\ell-3)/2\rceil$, i.e, $m-k+2\leq 2\lceil(\ell-3)/2\rceil$. 

By Lemma~\ref{lem:common-start}, $w_mv_\ell\notin E(G)$, since the paths $(v_{\ell+1},v_0)\odot v_0\rbjt{C}v_\ell$ and $v_{\ell+1}\rbjt{C}w_m$ have lengths at least $\ell$ and $k$, respectively. Moreover,   $\beta_{\ell-1-u}\le 
\beta_{\ell-1}+2\lceil u/2\rceil=
k-2+2\lceil u/2\rceil$ for all  $0\le u\le \ell-3$.
Since $\ell\geq 4$, $m-k+1\leq 2\lceil(\ell-3)/2\rceil-1\leq \ell-3$.

If $m-k+1$ is even, then $\beta_{\ell-1-u} \le m-1$ for $0\le u\le m-k+1$. Thus $w_m$ is nonadjacent to $v_\ell,v_{\ell-1},\ldots,v_{\ell-m+k-2}$, which are $m-k+3$ vertices. Hence $|N[w_m]|\le n-(m-k+3)=k+\ell-1$, contradicting $|N[w_m]|\ge k+\ell$. 

Now we suppose $m-k+1$ is odd. For $0\le u\le m-k$, we have $\beta_{\ell-1-u}\le m-1$. Thus $w_m$ is nonadjacent to $v_\ell,v_{\ell-1},\ldots,v_{\ell-m+k-1}$. Hence $|N[w_m]|\le n-(m-k+2)=k+\ell$. Since $|N[w_m]|\ge k+\ell$, these vertices are exactly the non-neighbors of $w_m$ in $G$. In particular, $v_{\ell+1}w_m\in E(G)$.

Next we prove $v_{0}w_1 \in E(G)$. Since  $\alpha_{i-d}\le \alpha_i+2$ for $d=1,2$ and $\alpha_2=m-k+3$,  $\alpha_{2+u}\ge m-k+3-2\lceil u/2\rceil$ for all $0\le u\le \ell-3$.
Since $m-k+2$ is even, for $0\le u\le m-k$, we have $\alpha_{2+u}\ge2$. Hence $w_1$ is nonadjacent to $v_2,v_3,\ldots,v_{m-k+2}$. By Lemma~\ref{lem:common-end}, 
it is also nonadjacent to $v_1$, since the paths $v_1\rbjt{C}v_{\ell+1}\odot(v_{\ell+1},v_0)$ and $w_1\rbjt{C}v_0$ have lengths at least $\ell$ and $k$, respectively. Thus $w_1$ is nonadjacent to $v_1,v_2,\ldots,v_{m-k+2}$. Hence $|N[w_1]|\le n-(m-k+2)=k+\ell$. Since $|N[w_1]|\ge k+\ell$, these vertices are exactly the non-neighbors of $w_1$ in $G$. In particular, $v_0w_1\in E(G)$.

If $m-k+2\ge4$, then $w_1\rbjt{C}w_m$ and $v_0\rbjt{C}v_{\ell+1}$ have lengths $m-1\ge k$ and $\ell+1\ge\ell$, respectively. Thus the existences of $v_0w_1$ and $v_{\ell+1}w_m$ contradict Lemma~\ref{lem:rectangle}.
If $m-k+2=2$, i.e., $m=k$,  then $n=k+\ell+2$ and  $(v_{\ell+1},v_0)$ is a $(k+1)$-dichord, a contradiction. This completes the proof.
\end{proof}

\begin{lemma}\label{lem:boundary-window}
Let $k\ge\ell+1 \ge 4$ and  $n=k+\ell+1$. If  $\delta(G)\ge k+\ell-1$, and $D$ is $C(k,\ell)$-free, then $N[v_i]=V(v_{i-\ell}\rbjt{C}v_{i+k-1})=V(D)\setminus\{v_{i+k}\}$ for every $0\le i\le n-1$. Moreover, $k=\ell+1$.
\end{lemma}

\begin{proof}
Since $\delta(G)\geq n-2$, each $v_i$ has at most one non-neighbor and $E(\overline G)$ is a matching. Without loss of generality, suppose to the contrary that  $v_0v_k \in E(G)$. By Lemma \ref{lem:Ll+1} and $n=k+\ell+1$, $(v_0,v_k) \in A(D)$. 

Suppose that $(v_i,v_{i+k})\in A(D)$ for $i\leq 0$. Since $n=k+\ell+1$ and $k\geq \ell+1\geq 4$, $v_i,v_{i-1+k},v_{i+k},v_{i-1+k+\ell}$ appear on $C$ in this order.
By Lemma \ref{lem:Ll+1}, $(v_{i-1+k+\ell},v_{i-1+k}) \notin A(D)$.
If $(v_{i-1+k},v_{i-1+k+\ell}) \in A(D)$, then $v_i\rbjt{C}v_{i-1+k} \odot (v_{i-1+k},v_{i-1+k+\ell})$ and $(v_i,v_{i+k}) \odot v_{i+k}\rbjt{C}v_{i-1+k+\ell}$ form a member of $C(k,\ell)$.   Thus, $v_{i-1+k}v_{i-1+k+\ell} \notin E(G)$ and hence $v_{i-1}v_{i-1+k} \in E(G)$. By Lemma \ref{lem:Ll+1}, we have $(v_{i-1},v_{i-1+k}) \in A(D)$.

%By Lemma \ref{lem:Ll+1}, $(v_{k+\ell-1},v_{k-1}) \notin A(D)$.
%If $(v_{k-1},v_{k+\ell-1}) \in A(D)$, then $v_0\rbjt{C}v_{k-1} \odot (v_{k-1},v_{k+\ell-1})$ and $(v_0,v_k) \odot v_k\rbjt{C}v_{k+\ell-1}$ form a member of $C(k,\ell)$.   Thus, $v_{k-1}v_{k+\ell-1} \notin E(G)$ and hence $v_{n-1}v_{k-1} \in E(G)$. By Lemma \ref{lem:Ll+1}, we have $(v_{n-1},v_{k-1}) \in A(D)$. 

By carrying out the above iteration for $i=0,\ldots, -n+1$, we have $v_{i-1+k}v_{i-1+k+\ell} \notin E(G)$ and 
$(v_{i},v_{i+k})\in A(D)$ for every $0\leq i\leq n-1$. The first result deduces that $v_jv_{j+\ell}\notin E(G)$ for every  $0\le j\le n-1$.
Since $n=k+\ell+1 >2\ell$, all $v_jv_{j+\ell}$ are distinct, which contradicts $E(\overline G)$ is a matching. 
Therefore, $N[v_i]=V(v_{i-\ell}\rbjt{C}v_{i+k-1})=V(D)\setminus\{v_{i+k}\}$ for every $0\le i\le n-1$.
If $k\ge\ell+2$, then for $0\le i\le n-1$, $v_iv_{i+k}$ are $n$ different non-edges since $n=k+\ell+1 < 2k$, which contradicts $E(\overline G)$ is a matching. Hence $k=\ell+1$. 
%Therefore $k=\ell+1$. Then $n=2(\ell+1)$, and each orientation of the edge $v_iv_{i+k}$ is an $(\ell+1)$-dichord. By Lemma~\ref{lem:Ll+1}, $v_iv_{i+k} \notin E(G)$. Since $|N[v_i]|\ge n-1$, this is the unique non-neighbor of $v_i$. The directed interval $v_{i-\ell}\rbjt{C}v_{i+k-1}$ contains every vertex except $v_{i+k}$, proving Lemma~\ref{lem:boundary-window}.
\end{proof}

\begin{lemma} \label{lem:Lkl,r}
Let  $k \geq \ell+1 \geq 4$ and  $n \ge k+\ell+2$. If  $\delta(G)\ge k+\ell-1$ and $D$ is $C(k,\ell)$-free, then $N[v_i]=V(v_{i-\ell}\rbjt{C}v_{i+k}) \setminus \{v_{i+\ell+1}\}$ for every $0\leq i\leq n-1$.
\end{lemma} 

\begin{proof}
By Lemma \ref{lem:Ll+1}, $v_iv_{i+\ell+1}\notin E(G)$. Since $|V(v_{i-\ell}\rbjt{C}v_{i+k}) \setminus \{v_{i+\ell+1}\}|=k+\ell\leq |N[v_i]|$, it suffices to show that $N[v_i]\subseteq V(v_{i-\ell}\rbjt{C}v_{i+k}) \setminus \{v_{i+\ell+1}\}$.

Suppose to the contrary that there exists $v_iv_j\in E(G)$ with $k+1 \le |v_i\rbjt{C}v_j|\le n-\ell-1$. Choose such an edge so that $|v_i\rbjt{C}v_j|$ is as small as possible, say $v_0v_r$. Then $k+1\leq r\leq n-\ell-1$.
Let $n=k+\ell+2+s, s\ge 0$. By Lemma \ref{lem:Ll+1},  $r\notin \{k+1,n-\ell-1\}$, and therefore $k+2\le r\le n-\ell-2=k+s$. Then $s\geq 2$ and $v_0,v_1,\ldots, v_{r+\ell+1}$ are distinct vertices.
To reach the final contradiction, we prove a series of assertions (1.1)--(1.6) as follows.

\textbf{(1.1)} $N[v_{r-1}]=V(v_{r-k-1}\rbjt{C}v_{r+\ell-1})\setminus \{v_{r-\ell-2}\}$ and $(v_{r-1},v_{r+\ell-1}) \in A(D)$.

 By Lemma \ref{lem:Ll+1}, $v_{r-\ell-2}v_{r-1} \notin E(G)$.
By the choice of $v_0v_r$,  $v_{r-1}$ has no neighbors in $V(v_0\rbjt{C}v_{r-k-2})$. For $r+\ell \le i \le n-1$, note that $v_0\rbjt{C}v_{r-1}$ and $v_{r}\rbjt{C}v_{i}$ have lengths at least $r-1 \geq k $ and $\ell$, respectively, and $v_0v_r \in E(G)$, by Lemma \ref{lem:rectangle} we have $v_{r-1}v_i \notin E(G)$. Since $k \geq \ell+1$,  $v_{r-\ell-2} \in V(v_{r-k-1}\rbjt{C}v_{r+\ell-1})$ and hence $N[v_{r-1}] \subseteq V(v_{r-k-1}\rbjt{C}v_{r+\ell-1}) \setminus \{v_{r-\ell-2}\}$, where $|V(v_{r-k-1}\rbjt{C}v_{r+\ell-1}) \setminus \{v_{r-\ell-2}\}|=k+\ell$. Therefore, $N[v_{r-1}] = V(v_{r-k-1}\rbjt{C}v_{r+\ell-1}) \setminus \{v_{r-\ell-2}\}$ and $v_{r-1}v_{r+\ell-1} \in E(G)$.

If $(v_{r+\ell-1},v_{r-1}) \in A(D)$, $v_0\rbjt{C}v_{r-1}$ and $v_{r}\rbjt{C}v_{r+\ell-1} \odot (v_{r+\ell-1},v_{r-1})$ have lengths $r-1 \geq k$ and $\ell$, respectively. By Lemma \ref{lem:common-end},  $v_0v_r \notin E(G)$,
a contradiction. Therefore,  $(v_{r-1},v_{r+\ell-1}) \in A(D)$.

\textbf{(1.2)} $(v_{r},v_0) \in A(D)$. Moreover, $D$ contains no $r$-dichord.

If not, $v_0\rbjt{C}v_{r-1} \odot (v_{r-1},v_{r+\ell-1})$ and $(v_0,v_{r}) \odot v_{r}\rbjt{C}v_{r+\ell-1}$ form a member of $C(k,\ell)$, a contradiction. If $D$ contains an $r$-dichord $(v_j,v_{j+r})$ for some $j$, then we can relabel the Hamiltonian cycle so that this arc is $(v_0,v_r)$, a contradiction. Hence $D$ contains no $r$-dichord.

In (1.3)--(1.4), we  assume that $k \geq \ell+2$.
\vskip 2mm
\textbf{(1.3)}
$N[v_{r-2}]=V(v_{r-k-2}\rbjt{C}v_{r+\ell-2}) \setminus \{v_{r-\ell-3}\}$. 

By Lemma \ref{lem:Ll+1}, $v_{r-\ell-3}v_{r-2},v_{r-2}v_{r+\ell-1}  \notin E(G)$.
By the choice of $v_0v_r$,  $v_{r-2}$ has no neighbors in $V(v_0\rbjt{C}v_{r-k-3})$. If $v_{r-2}$ has neighbor $v_i$ in $V(v_{r+\ell}\rbjt{C}v_{n-1})$, $(v_r,v_0) \odot v_0\rbjt{C}v_{r-2}$ and $v_{r}\rbjt{C}v_{i}$ have lengths  $r-1 \geq k $ and $\ell$, respectively, contradicting Lemma \ref{lem:common-start}. 
Since $k \geq \ell+2$,  $v_{r-\ell-3} \in V(v_{r-k-2}\rbjt{C}v_{r+\ell-2})$ and hence $N[v_{r-2}] \subseteq V(v_{r-k-2}\rbjt{C}v_{r+\ell-2}) \setminus \{v_{r-\ell-3}\}$, where $|V(v_{r-k-2}\rbjt{C}v_{r+\ell-2}) \setminus \{v_{r-\ell-3}\}|=k+\ell$, which implies $N[v_{r-2}] = V(v_{r-k-2}\rbjt{C}v_{r+\ell-2}) \setminus \{v_{r-\ell-3}\}$.

By (1.1) and the argument above, when $k \geq \ell+2$,  $v_{r-\ell-3}v_{r-1},v_{r-\ell-4}v_{r-2} \in E(G)$ (see  Figure \ref{fig:gel+2}). But $v_{r-1}\rbjt{C}v_{r-\ell-4}$ and $v_{r-\ell-3}\rbjt{C}v_{r-2}$ have lengths $n-\ell-3 = k+s-1 \geq k$ and $\ell+1 >\ell$, respectively, which gives a member of $C(k,\ell)$ by Lemma \ref{lem:rectangle}, a contradiction.
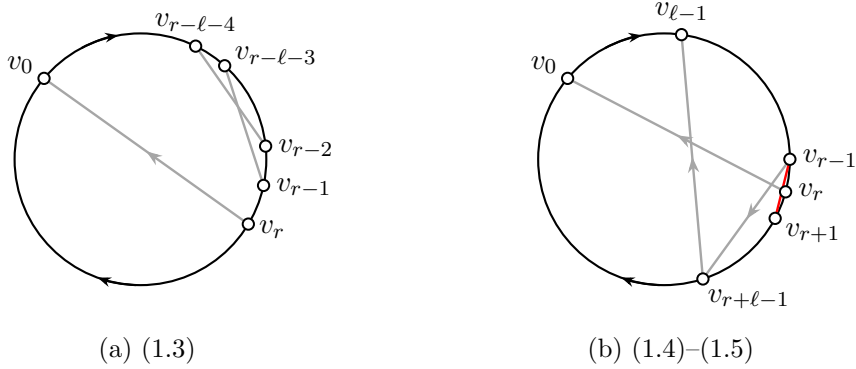
\begin{figure}[htbp]
\centering

% ---------- (1.3)----------
\begin{subfigure}[b]{0.45\textwidth}
\centering
\begin{tikzpicture}[
    scale=0.9,
    every node/.style={font=\normalfont\normalsize},
    vertex/.style={
        circle,
        draw=black,
        fill=white,
        line width=0.75pt,
        inner sep=1.5pt
    },
    edge/.style={line width=0.8pt},
    grayedge/.style={line width=0.9pt, gray!70},
    cwmark/.style={
        line width=0.8pt,
        -{Stealth[length=1.8mm,width=1.3mm]}
    },
    midarrow/.style={
        postaction={decorate},
        decoration={
            markings,
            mark=at position 0.5 with {
                \arrow{Stealth[length=2.1mm,width=1.5mm]}
            }
        }
    }
]

\def\R{1.85}

% fixed bounding box for alignment
\path[use as bounding box]
    (-2.45,-2.35) rectangle (2.65,2.35);

% vertices
\coordinate (v0) at (140:\R);
\coordinate (vrlfour) at (64:\R);
\coordinate (vrlthree) at (48:\R);
\coordinate (vrmii) at (6:\R);
\coordinate (vrmi) at (-12:\R);
\coordinate (vr) at (-31:\R);

% circle
\draw[edge] (0,0) circle (\R);

% small clockwise direction marks, placed away from vertices
\draw[cwmark]
    (118:\R)
    arc[start angle=118, end angle=100, radius=\R];

\draw[cwmark]
    (-92:\R)
    arc[start angle=-92, end angle=-110, radius=\R];

% chords
% midpoint arrow from v_r to v_0
\draw[grayedge,midarrow] (vr) -- (v0);

\draw[grayedge] (vrlfour) -- (vrmii);
\draw[grayedge] (vrlthree) -- (vrmi);

% vertices
\node[vertex] at (v0) {};
\node[vertex] at (vrlfour) {};
\node[vertex] at (vrlthree) {};
\node[vertex] at (vrmii) {};
\node[vertex] at (vrmi) {};
\node[vertex] at (vr) {};

% labels
\node[above left=0pt, yshift=-2pt] at (v0) {$v_0$};
\node[above=1pt] at (vrlfour) {$v_{r-\ell-4}$};
\node[above right=0pt,yshift=-4pt] at (vrlthree) {$v_{r-\ell-3}$};
\node[right=1pt,yshift=-1pt] at (vrmii) {$v_{r-2}$};
\node[right=1pt,yshift=-1pt] at (vrmi) {$v_{r-1}$};
\node[xshift=9pt,yshift=-2pt] at (vr) {$v_r$};

\end{tikzpicture}
\caption{(1.3)}
\label{fig:gel+2}
\end{subfigure}
\hspace{0.01\textwidth}
% ---------- (1.4)--(1.5) ----------
\begin{subfigure}[b]{0.45\textwidth}
\centering
\begin{tikzpicture}[
    scale=0.9,
    every node/.style={font=\normalfont\normalsize},
    vertex/.style={
        circle,
        draw=black,
        fill=white,
        line width=0.75pt,
        inner sep=1.5pt
    },
    edge/.style={line width=0.8pt},
    rededge/.style={line width=0.9pt, red},
    grayedge/.style={line width=0.9pt, gray!70},
    cwmark/.style={
        line width=0.8pt,
        -{Stealth[length=1.8mm,width=1.3mm]}
    },
    midarrow/.style={
        postaction={decorate},
        decoration={
            markings,
            mark=at position 0.5 with {
                \arrow{Stealth[length=2.1mm,width=1.5mm]}
            }
        }
    }
]

\def\R{1.85}

% fixed bounding box for alignment
\path[use as bounding box]
    (-2.45,-2.35) rectangle (2.65,2.35);

% vertices
\coordinate (v0) at (140:\R);
\coordinate (velli) at (82:\R);
\coordinate (vrmi) at (0:\R);
\coordinate (vr) at (-15:\R);
\coordinate (vrpi) at (-28:\R);
\coordinate (vrlelli) at (-72:\R);

% circle
\draw[edge] (0,0) circle (\R);

% small clockwise direction marks, placed away from vertices
\draw[cwmark]
    (118:\R)
    arc[start angle=118, end angle=100, radius=\R];

\draw[cwmark]
    (-92:\R)
    arc[start angle=-92, end angle=-110, radius=\R];

% chords
% midpoint arrow from v_r to v_0
\draw[grayedge,midarrow] (vr) -- (v0);

% midpoint arrow from v_{r+\ell-1} to v_{\ell-1}
\draw[grayedge,midarrow] (vrlelli) -- (velli);

% midpoint arrow from v_{r-1} to v_{r+\ell-1}
\draw[grayedge,midarrow] (vrmi) -- (vrlelli);

\draw[rededge] (vrmi) -- (vrpi);

% vertices
\node[vertex] at (v0) {};
\node[vertex] at (velli) {};
\node[vertex] at (vrmi) {};
\node[vertex] at (vr) {};
\node[vertex] at (vrpi) {};
\node[vertex] at (vrlelli) {};

% labels
\node[above left=0pt,yshift=-2pt] at (v0) {$v_0$};
\node[above=1pt] at (velli) {$v_{\ell-1}$};
\node[right=1pt] at (vrmi) {$v_{r-1}$};
\node[right=1pt] at (vr) {$v_r$};
\node[below right=0pt,yshift=2pt] at (vrpi) {$v_{r+1}$};
\node[below right=0pt,xshift=-2pt] at (vrlelli) {$v_{r+\ell-1}$};

\end{tikzpicture}
\caption{(1.4)--(1.5)}
\label{fig:l+1}
\end{subfigure}

\caption{Illustrations for Lemma~\ref{lem:Lkl,r}}
\label{fig:4}
\end{figure}

In (1.4)--(1.5), we assume that $k = \ell+1$.
\vskip 2mm
\textbf{(1.4)} $N[v_{\ell-1}]=V(v_{0}\rbjt{C}v_{2\ell-1}) \cup \{v_{r+\ell-1}\}$, $(v_{r+\ell-1},v_{\ell-1}) \in A(D)$.

By Lemma \ref{lem:Ll+1} and the choice of $v_0v_r$,  $v_{\ell-1}$ has no neighbors in $V(v_{2\ell}\rbjt{C}v_{r+\ell-2})$. If $v_{\ell-1}$ has neighbor $v_{r+\ell}$, since $(v_{r},v_0) \odot v_0\rbjt{C}v_{\ell-1}$ and $v_{r}\rbjt{C}v_{r+\ell}$ both have length  $\ell$, and since any orientation of $v_{\ell-1} v_{r+\ell}$ forces one of the two paths to have length $\ell+1$, and thus creating a $C(\ell+1,\ell)$, a contradiction.
If $v_{\ell-1}$ has neighbor $v_i$ in $V(v_{r+\ell+1}\rbjt{C}v_{n-1})$, $(v_{r},v_0) \odot v_0\rbjt{C}v_{\ell-1}$ and $v_{r}\rbjt{C}v_{i}$ have lengths  $\ell$ and $\ell+1$, respectively, a contradiction.
Hence $N[v_{\ell-1}] \subseteq V(v_{0}\rbjt{C}v_{2\ell-1}) \cup \{v_{r+\ell-1}\}$, where $|V(v_{0}\rbjt{C}v_{2\ell-1}) \cup \{v_{r+\ell-1}\}|=2\ell+1$, which implies $N[v_{\ell-1}] = V(v_{0}\rbjt{C}v_{2\ell-1}) \cup \{v_{r+\ell-1}\}$.
Note that $(v_{\ell-1},v_{r+\ell-1})$ is an $r$-dichord. By (1.2), $(v_{r+\ell-1},v_{\ell-1})\in A(D)$.

\textbf{(1.5)} $v_{r-1}v_{r+1}\notin E(G)$. 

Since $k=\ell+1$, by (1.1), $N[v_{r-1}]=V(v_{r-\ell-1}\rbjt{C}v_{r+\ell-1})$ and $(v_{r-1},v_{r+\ell-1})\in A(D)$. By (1.4), $(v_{r+\ell-1},v_{\ell-1})\in A(D)$. As shown in Figure \ref{fig:l+1}: If $(v_{r-1},v_{r+1})\in A(D)$, then $(v_{r-1},v_r)\odot(v_r,v_0)\odot v_0\rbjt{C}v_{\ell-1}$ and $(v_{r-1},v_{r+1})\odot v_{r+1}\rbjt{C}v_{r+\ell-1}\odot(v_{r+\ell-1},v_{\ell-1})$ form a member of $C(\ell+1,\ell)$, a contradiction; If $(v_{r+1},v_{r-1})\in A(D)$, choose the smallest $i$ with $r+2\le i\le r+\ell-1$ such that $(v_{i-1},v_{r-1}),(v_{r-1},v_i)\in A(D)$. Such a vertex exists because $(v_{r+1},v_{r-1}),(v_{r-1},v_{r+\ell-1})\in A(D)$. Then $v_r\rbjt{C}v_{i-1}\odot(v_{i-1},v_{r-1})\odot(v_{r-1},v_i)\odot v_i\rbjt{C}v_{r+\ell-1}\odot(v_{r+\ell-1},v_{\ell-1})$ and $(v_r,v_0)\odot v_0\rbjt{C}v_{\ell-1}$ form a member of $C(\ell+1,\ell)$, a contradiction.

But (1.5) contradicts $v_{r+1}\in N[v_{r-1}]$ in (1.1).
These contradictions show that $N[v_i]\subseteq V(v_{i-\ell}\rbjt{C}v_{i+k}) \setminus \{v_{i+\ell+1}\}$ for $k\ge \ell+1$, completing the proof.
\end{proof}

\begin{proof}[\bfseries{Proof of Theorem \ref{thm:Ll+1}}]
Assume that $D$ is $C(k,\ell)$-free.
Suppose $k\ge \ell+2$. By Lemma~\ref{lem:boundary-window}, $n\ge k+\ell+2$. Then Lemma~\ref{lem:Lkl,r} indicates
$N[v_i]=V(v_{i-\ell}\rbjt{C}v_{i+k})\setminus\{v_{i+\ell+1}\}$ for every $0\le i\le n-1$.
Moreover, Lemma~\ref{lem:Ll} forbids anti-$\ell$-dichords and anti-$k$-dichords. Hence $(v_0,v_k),(v_{k-1},v_{k+\ell-1})\in A(D)$ (see in Figure \ref{fig:6_1}). Therefore, $v_0\rbjt{C}v_{k-1} \odot (v_{k-1},v_{k+\ell-1})$ and $(v_0,v_k) \odot v_k\rbjt{C}v_{k+\ell-1}$ form a member of $C(k,\ell)$, a contradiction.

\begin{figure}[htbp]
\centering

% ---------- first figure ----------
\begin{subfigure}[b]{0.45\textwidth}
\centering
\begin{tikzpicture}[
    scale=0.9,
    every node/.style={font=\normalfont\normalsize},
    vertex/.style={
        circle,
        draw=black,
        fill=white,
        line width=0.75pt,
        inner sep=1.5pt
    },
    edge/.style={line width=0.8pt},
    rededge/.style={line width=0.9pt, red},
    grayedge/.style={line width=0.9pt, gray!70},
    cwmark/.style={
        line width=0.8pt,
        -{Stealth[length=1.8mm,width=1.3mm]}
    },
    midarrow/.style={
        postaction={decorate},
        decoration={
            markings,
            mark=at position 0.5 with {
                \arrow{Stealth[length=2.1mm,width=1.5mm]}
            }
        }
    }
]

\def\R{1.85}

% fixed bounding box for alignment
\path[use as bounding box]
    (-2.45,-2.35) rectangle (2.65,2.35);

% vertices
\coordinate (v0) at (138:\R);
\coordinate (vlmi) at (42:\R);
\coordinate (vl) at (30:\R);
\coordinate (vklli) at (-58:\R);

% circle
\draw[edge] (0,0) circle (\R);

% chords
% midpoint arrow from v_0 to v_k
\draw[grayedge,midarrow] (v0) -- (vl);

% midpoint arrow from v_{k-1} to v_{k+\ell-1}
\draw[grayedge,midarrow] (vlmi) -- (vklli);

% vertices
\node[vertex] at (v0) {};
\node[vertex] at (vlmi) {};
\node[vertex] at (vl) {};
\node[vertex] at (vklli) {};

% small clockwise direction marks, drawn on top
\draw[cwmark]
    (112:\R)
    arc[start angle=112, end angle=94, radius=\R];

\draw[cwmark]
    (-12:\R)
    arc[start angle=-12, end angle=-30, radius=\R];

% labels
\node[above left=0pt,yshift=-4pt] at (v0) {$v_0$};

% moved further upward
\node[
    anchor=south,
    inner sep=0pt,
    xshift=1pt,
    yshift=6pt
] at (vlmi) {$v_{k-1}$};

\node[right=1pt] at (vl) {$v_k$};
\node[below right=0pt,yshift=1pt] at (vklli) {$v_{k+\ell-1}$};

\end{tikzpicture}
\caption{2-dichord, 3-dichord}
\label{fig:6_1}
\end{subfigure}
\hspace{0.01\textwidth}
% ---------- second figure ----------
\begin{subfigure}[b]{0.45\textwidth}
\centering
\begin{tikzpicture}[
    scale=0.9,
    every node/.style={font=\normalfont\normalsize},
    vertex/.style={
        circle,
        draw=black,
        fill=white,
        line width=0.75pt,
        inner sep=1.5pt
    },
    edge/.style={line width=0.8pt},
    rededge/.style={line width=0.9pt, red},
    grayedge/.style={line width=0.9pt, gray!70},
    cwmark/.style={
        line width=0.8pt,
        -{Stealth[length=1.8mm,width=1.3mm]}
    },
    midarrow/.style={
        postaction={decorate},
        decoration={
            markings,
            mark=at position 0.5 with {
                \arrow{Stealth[length=2.1mm,width=1.5mm]}
            }
        }
    }
]

\def\R{1.85}

% fixed bounding box for alignment
\path[use as bounding box]
    (-2.45,-2.35) rectangle (2.65,2.35);

% vertices
\coordinate (v0) at (138:\R);
\coordinate (vlmi) at (42:\R);
\coordinate (vl) at (30:\R);
\coordinate (v2lmii) at (-42:\R);
\coordinate (v2lmi) at (-58:\R);
\coordinate (v2l) at (-76:\R);
\coordinate (v2lp) at (-100:\R);

% circle
\draw[edge] (0,0) circle (\R);

% chords
% midpoint arrow from v_0 to v_\ell
\draw[grayedge,midarrow] (v0) -- (vl);

% midpoint arrow from v_{\ell-1} to v_{2\ell-1}
\draw[grayedge,midarrow] (vlmi) -- (v2lmi);

\draw[rededge] (v2lp) -- (v2lmii);
\draw[rededge] (v2l) -- (v2lmii);

% vertices
\node[vertex] at (v0) {};
\node[vertex] at (vlmi) {};
\node[vertex] at (vl) {};
\node[vertex] at (v2lmii) {};
\node[vertex] at (v2lmi) {};
\node[vertex] at (v2l) {};
\node[vertex] at (v2lp) {};

% small clockwise direction marks, drawn on top
\draw[cwmark]
    (112:\R)
    arc[start angle=112, end angle=94, radius=\R];

\draw[cwmark]
    (-12:\R)
    arc[start angle=-12, end angle=-30, radius=\R];

% labels
\node[above left=0pt,yshift=-4pt] at (v0) {$v_0$};

% moved further upward
\node[
    anchor=south,
    inner sep=0pt,
    xshift=1pt,
    yshift=6pt
] at (vlmi) {$v_{\ell-1}$};

\node[right=1pt] at (vl) {$v_\ell$};
\node[right=1pt,yshift=-3pt] at (v2lmii) {$v_{2\ell-2}$};
\node[below right=0pt,yshift=2pt] at (v2lmi) {$v_{2\ell-1}$};
\node[below=1pt] at (v2l) {$v_{2\ell}$};

\node[
    anchor=north east,
    inner sep=0pt,
    xshift=3pt,
    yshift=-3pt
] at (v2lp) {$v_{2\ell+1}$};

\end{tikzpicture}
\caption{2-dichord, 3-dichord}
\label{fig:6_2}
\end{subfigure}
\caption{Illustrations for Theorem~\ref{thm:Ll+1}}
\label{fig:5}
\end{figure}
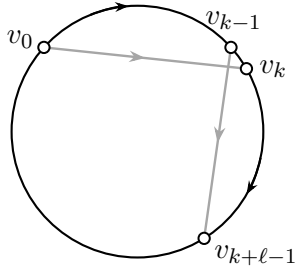
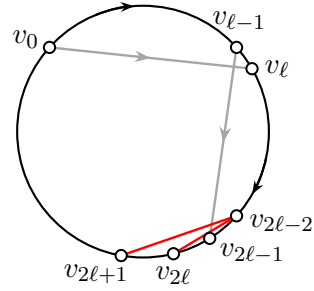
Now we suppose $k=\ell+1$. By Lemma~\ref{lem:boundary-window} when $n=k+\ell+1$, and by
Lemma~\ref{lem:Lkl,r} when $n\ge k+\ell+2$, for every $0\leq i\leq n-1$, $v_i$ is adjacent to $v_{i-\ell}$,
$v_{i+2}$, $v_{i+3}$, $v_{i+\ell}$. 
We claim that $C$ contains neither  $2$-dichords nor  $3$-dichords.
By cyclic symmetry, suppose to the contrary that $(v_{2\ell-2},v_{2\ell-2+t})\in A(D)$ with $t\in\{2,3\}$. Since $n\geq k+\ell+1=2\ell+2$, we have $2\ell-2+t\leq n-1$.
By Lemma~\ref{lem:Ll+1} when $n=k+\ell+1$, $C$ contains no $(k+1)$-dichords, i.e., anti-$\ell$-dichords, and by Lemma~\ref{lem:Ll} when $n\ge k+\ell+2$, $C$ contains no anti-$\ell$-dichords. Hence $(v_0,v_\ell),(v_{\ell-1},v_{2\ell-1})\in A(D)$. As shown in Figure \ref{fig:6_2}: $v_0\rbjt{C}v_{\ell-1}\odot(v_{\ell-1},v_{2\ell-1})\odot v_{2\ell-1}\rbjt{C}v_{2\ell-2+t}$ and $(v_0,v_\ell)\odot v_\ell\rbjt{C}v_{2\ell-2}\odot(v_{2\ell-2},v_{2\ell-2+t})$ have lengths at least $\ell+1$ and $\ell$, respectively, they form a member of $C(\ell+1,\ell)$, a contradiction. Hence, no 2-dichords or 3-dichords exist. It follows that $(v_{i+2},v_i),(v_{i+3},v_i)\in A(D)$ for every $0\leq i\leq n-1$.

%We show that $C$ contains neither 2-dichords nor 3-dichords. Suppose to the contrary that $(v_{2\ell-2},v_{2\ell-2+t})\in A(D)$ with $t\in\{2,3\}$. Since $n\geq k+\ell+1=2\ell+2$, we have $2\ell-2+t\leq n-1$.
%When $n=2\ell+2$, by Lemma~\ref{lem:boundary-window}, $N[v_i]=V(v_{i-\ell}\rbjt{C}v_{i+k-1})=V(v_{i-\ell}\rbjt{C}v_{i+\ell})$ for $0\leq i\leq n-1$. By Lemma~\ref{lem:Ll+1}, $C$ contains no $(k+1)$-dichords, i.e. anti-$\ell$-dichords. When $n\geq2\ell+3$, by Lemma~\ref{lem:Lkl,r}, $N[v_i]=V(v_{i-\ell}\rbjt{C}v_{i+\ell})$ for $0\leq i\leq n-1$. By Lemma~\ref{lem:Ll}, $C$ contains no anti-$\ell$-dichords. Hence $(v_0,v_\ell),(v_{\ell-1},v_{2\ell-1})\in A(D)$.
%Since $v_0\rbjt{C}v_{\ell-1}\odot(v_{\ell-1},v_{2\ell-1})\odot v_{2\ell-1}\rbjt{C}v_{2\ell-2+t}$ and $(v_0,v_\ell)\odot v_\ell\rbjt{C}v_{2\ell-2}\odot(v_{2\ell-2},v_{2\ell-2+t})$ have lengths at least $\ell+1$ and $\ell$, respectively, they form a member of $C(\ell+1,\ell)$, a contradiction. Hence no 2-dichords or 3-dichords exist.

Since $n\geq 2\ell+2$, the directed path $P_1=v_{2\ell+1}\rbjt{C}v_{\ell-1}$ has length $n-(2\ell+1)+(\ell-1)=n-\ell-2\geq\ell$.
We next construct a second directed path from $v_{2\ell+1}$ to $v_{\ell-1}$ whose internal vertices lie in $V(v_\ell\rbjt{C}v_{2\ell})$. Note that $|V(v_{\ell-1}\rbjt{C}v_{2\ell+1})|=\ell+3$.

If $\ell+1\equiv 0\pmod 3$, let
\[P_2=v_{2\ell+1}v_{2\ell-1}v_{2\ell}v_{2\ell-2}\cdots v_{\ell+3j}v_{\ell+3j-2}v_{\ell+3j-1}v_{\ell+3j-3}\cdots v_{\ell+3}v_{\ell+1}v_{\ell+2}v_{\ell-1}.\]

If $\ell+1\equiv 1\pmod 3$, let
\[P_2=v_{2\ell+1}v_{2\ell-1}v_{2\ell}v_{2\ell-2}\cdots v_{\ell+3j+1}v_{\ell+3j-1}v_{\ell+3j}v_{\ell+3j-2}\cdots v_{\ell+1}v_{\ell-1}.\]

If $\ell+1\equiv 2\pmod 3$, let
\[P_2=v_{2\ell+1}v_{2\ell-1}v_{2\ell}v_{2\ell-2}\cdots v_{\ell+3j+2}v_{\ell+3j}v_{\ell+3j+1}v_{\ell+3j-1}\cdots v_{\ell+2}v_{\ell}v_{\ell+1}v_{\ell-1}.\]

It is easy to check that 
$V(v_{\ell-1}\rbjt{C}v_{2\ell+1})\backslash V(P_2)=\{v_\ell\}$ when $\ell+1\not\equiv 2\pmod{3}$, and  $V(v_{\ell-1}\rbjt{C}v_{2\ell+1})=V(P_2)$ when $\ell+1\equiv 2 \pmod 3$. Thus, in all cases, $|P_2|\geq \ell+1$.
Hence $P_1$ and $P_2$ form a member of $C(\ell+1,\ell)$, contradicting the assumption that $D$ is $C(\ell+1,\ell)$-free.
\end{proof}

In the remainder of this section, we prove Theorem \ref{thm:Ll}. For each $0\le i\le n-1$, let $F_i=v_iv_{i+\ell}$, where the subscripts are taken modulo $n$.

\begin{lemma}\label{lem:antil}
Let $\ell\ge3$ and $n\ge2\ell+1$. If $\delta(G)\ge2\ell-1$ and $D$ is $C(\ell,\ell)$-free, then for each $0\le i\le n-1$,  $F_i$ and $F_{i+\ell-1}$ cannot both exist.
\end{lemma}
\begin{proof}

When $n=2\ell+1$, $(v_{i+\ell},v_i)$ would be an
$(\ell+1)$-dichord, which is forbidden by Lemma~\ref{lem:Ll+1}. When
$n\ge2\ell+2$, $(v_{i+\ell},v_i)$ would be an anti-$\ell$-dichord,
which is forbidden by Lemma~\ref{lem:Ll}. Thus, whenever $F_i$ exists, it must
be oriented as $(v_i,v_{i+\ell})$.

If $F_i,F_{i+\ell-1}$ both exist, the paths
$v_i\rbjt{C}v_{i+\ell-1} \odot (v_{i+\ell-1},v_{i+2\ell-1})$ and $(v_i,v_{i+\ell}) \odot v_{i+\ell}\rbjt{C}v_{i+2\ell-1}$
form a member of $C(\ell,\ell)$, a contradiction.
\end{proof}

We call an edge $v_iv_j\in E(G)$ \emph{long} if $\theta(v_i,v_j)\geq \ell+1$. Otherwise, call it \emph{local}. If $v_iv_j$ is a long (resp. local) edge, we also say that $v_j$ is a long (resp. local) neighbor of $v_i$.

\begin{lemma}\label{lem:bal-no-long}
Let $\ell\ge3$ and $n\ge2\ell+2$. If $\delta(G)\ge2\ell-1$ and $D$ is $C(\ell,\ell)$-free, then $G$ contains no long edge, i.e., $N[v_i]\subseteq V(v_{i-\ell}\rbjt{C}v_{i+\ell})$ for every $0\leq i\leq n-1$.
\end{lemma}

\begin{proof}
Suppose to the contrary that there exists $v_iv_j\in E(G)$ with $\ell+1 \le |v_i\rbjt{C}v_j| \le n-\ell-1$. Choose such an edge so that $|v_i\rbjt{C}v_j|$ is as small as possible, say $v_0v_r$. Then $\ell+1\leq r\leq n-\ell-1$ and  $\ell+1 \le n-r \le n-\ell-1$, and by the minimality of $r$, we have $r\le n-r$. By Lemma~\ref{lem:Ll+1}, $\ell+2\le r\le n-\ell-2$.
Let $r=\ell+s$. Then  $2\leq s\le n-2\ell-2$.
Since $r \le n-\ell-2$, $v_0,v_1,\ldots,v_{r+\ell+1}$ are distinct vertices. To reach the final contradiction, we prove a series of assertions (2.1)--(2.3) as follows.

\medskip  
\textbf{(2.1)} $(v_r,v_0)\in A(D)$. Moreover, $C$ contains no $r$-dichord.

Suppose to the contrary that $(v_0,v_r)\in A(D)$. Consider $N[v_{r-1}]$. If $0\le j\le r-\ell-2$ and $v_jv_{r-1}\in E(G)$, then $\ell+1\le r-1-j\le r-1$, contradicting the choice of $r$. If $j\ge r+\ell-1$ , then $v_0\rbjt{C}v_{r-1}$ and $(v_0,v_r) \odot v_r\rbjt{C}v_j$ have lengths at least $r-1\ge \ell$ and $\ell$, respectively. By Lemma \ref{lem:common-start}, $v_jv_{r-1}\notin E(G)$. Thus $N[v_{r-1}]\subseteq V(v_{r-\ell-1}\rbjt{C}v_{r+\ell-2})$, and hence $N[v_{r-1}]=V(v_{r-\ell-1}\rbjt{C}v_{r+\ell-2})$. In particular, $v_{r-1}v_{r+\ell-2}\in E(G)$. If $(v_{r+\ell-2},v_{r-1})\in A(D)$, then $v_0\rbjt{C}v_{r-1}$ and $(v_0,v_r)\odot v_r\rbjt{C}v_{r+\ell-2}\odot (v_{r+\ell-2},v_{r-1})$ form a member of $C(\ell,\ell)$, a contradiction. Hence $(v_{r-1},v_{r+\ell-2})\in A(D)$.

Next we determine $N[v_1]$. By the choice of $v_0v_r$,  $v_1v_j\notin E(G)$ for $\ell+2\le j\le r$. If $r+1\le j\le n-\ell+1$, note that $v_j\rbjt{C}v_0\odot (v_0,v_r)$ and $v_1\rbjt{C}v_r$ have lengths $n-j+1\ge \ell$ and $r-1\ge \ell$, respectively. By Lemma \ref{lem:common-end},  $v_1v_j\notin E(G)$. Then $N[v_1]\subseteq V(v_{n-\ell+2}\rbjt{C}v_{\ell+1})$, where $|V(v_{n-\ell+2}\rbjt{C}v_{\ell+1})|=2\ell$. Thus $N[v_1]=V(v_{n-\ell+2}\rbjt{C}v_{\ell+1})$.

Now $v_1\rbjt{C}v_{r-1}\odot (v_{r-1},v_{r+\ell-2})$ and $(v_{n-1},v_0)\odot (v_0,v_r)\odot v_r\rbjt{C}v_{r+\ell-2}$ have lengths $r-1\ge \ell$ and $\ell$, respectively. By Lemma \ref{lem:common-end}, $v_1v_{n-1}\notin E(G)$, contradicting $N[v_1]=V(v_{n-\ell+2}\rbjt{C}v_{\ell+1})$. Therefore $(v_0,v_r)\notin A(D)$ and $(v_r,v_0)\in A(D)$.

If $(v_j,v_{j+r})\in A(D)$ for some $j$, we can relabel the Hamiltonian cycle so that this arc is $(v_0,v_r)$. The argument above gives the same contradiction. Hence $C$ contains no $r$-dichord.

\medskip
\textbf{(2.2)} $F_{n-1}\notin E(G)$ and $N[v_{\ell-1}]=V(v_0\rbjt{C}v_{2\ell-1})$.

If $F_{n-1}$ exists,  by Lemma \ref{lem:Ll},  $(v_{n-1},v_{\ell-1})\in A(D)$. Then $(v_r,v_0)\odot v_0\rbjt{C}v_{\ell-1}$ and $v_r\rbjt{C}v_{n-1}\odot (v_{n-1},v_{\ell-1})$ form a member of $C(\ell,\ell)$, a contradiction. Thus $F_{n-1}\notin E(G)$ and $v_{n-1}\notin N[v_{\ell-1}]$.

Suppose $v_{\ell-1}v_j\in E(G)$ with $2\ell\le j \le n-2$. Since $\ell+1\le j-\ell+1\le n-\ell-1$, by the choice of $v_0v_r$,  $r\le j-\ell+1$. Then $r<r+\ell-1\le j$. Since  $\ell-1<r<j<n-1$, $v_0,v_{\ell-1},v_r,v_j$ appear in this order, as shown in Figure \ref{fig:8_1}. Since $|v_0\rbjt{C}v_{\ell-1}|=\ell-1$ and $|v_r\rbjt{C}v_j|=j-r\ge \ell-1$, by Lemma \ref{lem:boundary-crossing} and $(v_r,v_0)\in A(D)$, $D$ is  $C(\ell,\ell)$-free  only if $|v_r\rbjt{C}v_j|=\ell-1$ and $(v_{\ell-1},v_j)\in A(D)$. Hence $j-\ell+1=r$, which implies $(v_{\ell-1},v_j)$ is an $r$-dichord, contradicting (2.1). Hence  $N[v_{\ell-1}]\subseteq V(v_0\rbjt{C}v_{2\ell-1})$ and  $N[v_{\ell-1}]=V(v_0\rbjt{C}v_{2\ell-1})$. %In particular, $F_{\ell-1}\in E(G)$.

\begin{figure}[htbp]
\centering

% ---------- (2.2) ----------
\begin{subfigure}[b]{0.31\textwidth}
\centering
\begin{tikzpicture}[
    scale=1.25,
    every node/.style={font=\normalfont\normalsize},
    vertex/.style={
        circle,
        draw=black,
        fill=white,
        line width=0.75pt,
        inner sep=1.5pt
    },
    edge/.style={line width=0.65pt},
    grayedge/.style={line width=0.75pt, gray!70},
    cwmark/.style={
        line width=0.65pt,
        -{Stealth[length=1.5mm,width=1.05mm]}
    },
    midarrow/.style={
        postaction={decorate},
        decoration={
            markings,
            mark=at position 0.5 with {
                \arrow{Stealth[length=1.8mm,width=1.25mm]}
            }
        }
    }
]

\def\R{1.35}

\path[use as bounding box] (-1.85,-1.75) rectangle (2.05,1.75);

% vertices
\coordinate (v0) at (135:\R);
\coordinate (velli) at (88:\R);
\coordinate (vr) at (0:\R);
\coordinate (v2ellmi) at (-14:\R);
\coordinate (vj) at (-72:\R);

% circle
\draw[edge] (0,0) circle (\R);

% chords
\draw[grayedge,midarrow] (vr) -- (v0);
\draw[grayedge] (velli) -- (vj);

% vertices
\node[vertex] at (v0) {};
\node[vertex] at (velli) {};
\node[vertex] at (vr) {};
\node[vertex] at (v2ellmi) {};
\node[vertex] at (vj) {};

% clockwise direction marks, drawn on top
\draw[cwmark] (78:\R) arc[start angle=78, end angle=61, radius=\R];
\draw[cwmark] (-105:\R) arc[start angle=-105, end angle=-122, radius=\R];

% labels
\node[above left=0pt,yshift=-2pt] at (v0) {$v_0$};
\node[above=1pt] at (velli) {$v_{\ell-1}$};
\node[right=1pt] at (vr) {$v_r$};
\node[right=1pt,yshift=-2pt] at (v2ellmi) {$v_{2\ell-1}$};
\node[below right=0pt,xshift=-2pt] at (vj) {$v_j$};

\end{tikzpicture}
\caption{}
\label{fig:8_1}
\end{subfigure}
\hfill
% ---------- (2.4) ----------
\begin{subfigure}[b]{0.31\textwidth}
\centering
\begin{tikzpicture}[
    scale=1.25,
    every node/.style={font=\normalfont\normalsize},
    vertex/.style={
        circle,
        draw=black,
        fill=white,
        line width=0.75pt,
        inner sep=1.5pt
    },
    edge/.style={line width=0.65pt},
    grayedge/.style={line width=0.75pt, gray!70},
    cwmark/.style={
        line width=0.65pt,
        -{Stealth[length=1.5mm,width=1.05mm]}
    },
    midarrow/.style={
        postaction={decorate},
        decoration={
            markings,
            mark=at position 0.5 with {
                \arrow{Stealth[length=1.8mm,width=1.25mm]}
            }
        }
    }
]

\def\R{1.35}

\path[use as bounding box] (-1.85,-1.75) rectangle (2.05,1.75);

% vertices
\coordinate (v0) at (140:\R);
\coordinate (vij) at (72:\R);
\coordinate (vi) at (42:\R);
\coordinate (vr) at (10:\R);
\coordinate (alphai) at (-42:\R);
\coordinate (betaij) at (-140:\R);

% circle
\draw[edge] (0,0) circle (\R);

% chords
\draw[grayedge,midarrow] (vr) -- (v0);
\draw[grayedge] (vij) -- (betaij);
\draw[grayedge] (vi) -- (alphai);

% vertices
\node[vertex] at (v0) {};
\node[vertex] at (vij) {};
\node[vertex] at (vi) {};
\node[vertex] at (vr) {};
\node[vertex] at (alphai) {};
\node[vertex] at (betaij) {};

% clockwise direction marks, drawn on top
\draw[cwmark] (68:\R) arc[start angle=68, end angle=52, radius=\R];
\draw[cwmark] (5:\R) arc[start angle=5, end angle=-11, radius=\R];

% labels
\node[above left=0pt,yshift=-2pt] at (v0) {$v_0$};
\node[above=1pt] at (vij) {$v_{i-j}$};
\node[above right=0pt,yshift=-2pt] at (vi) {$v_i$};
\node[right=1pt] at (vr) {$v_r$};
\node[below right=0pt,yshift=1pt] at (alphai) {$\alpha_i$};
\node[below left=0pt,yshift=4pt] at (betaij) {$\beta_{i-j}$};

\end{tikzpicture}
\caption{}
\label{fig:8_2}
\end{subfigure}
\hfill
% ---------- (2.3) ----------
\begin{subfigure}[b]{0.31\textwidth}
\centering
\begin{tikzpicture}[
    scale=1.25,
    every node/.style={font=\normalfont\normalsize},
    vertex/.style={
        circle,
        draw=black,
        fill=white,
        line width=0.75pt,
        inner sep=1.5pt
    },
    edge/.style={line width=0.65pt},
    rededge/.style={line width=0.9pt, red},
    grayedge/.style={line width=0.75pt, gray!70},
    cwmark/.style={
        line width=0.65pt,
        -{Stealth[length=1.5mm,width=1.05mm]}
    },
    midarrow/.style={
        postaction={decorate},
        decoration={
            markings,
            mark=at position 0.5 with {
                \arrow{Stealth[length=1.8mm,width=1.25mm]}
            }
        }
    }
]

\def\R{1.35}

\path[use as bounding box] (-1.85,-1.75) rectangle (2.05,1.75);

% vertices
\coordinate (v0) at (135:\R);
\coordinate (vsp) at (48:\R);
\coordinate (vr) at (0:\R);
\coordinate (wml) at (-172:\R);

% circle
\draw[edge] (0,0) circle (\R);

% chords
\draw[grayedge,midarrow] (vr) -- (v0);
\draw[rededge] (vsp) -- (wml);

% vertices
\node[vertex] at (v0) {};
\node[vertex] at (vsp) {};
\node[vertex] at (vr) {};
\node[vertex] at (wml) {};

% clockwise direction marks, drawn on top
\draw[cwmark] (70:\R) arc[start angle=70, end angle=53, radius=\R];
\draw[cwmark] (-105:\R) arc[start angle=-105, end angle=-122, radius=\R];

% labels
\node[above left=0pt,yshift=-2pt] at (v0) {$v_0$};
\node[above right=0pt,yshift=-3pt] at (vsp) {$v_{s+1}$};
\node[right=1pt] at (vr) {$v_r$};
\node[below left=0pt,yshift=3pt] at (wml) {$w_{m-\ell+2}$};

\end{tikzpicture}
\caption{}
\label{fig:8_3}
\end{subfigure}

\caption{Illustrations for Lemma \ref{lem:bal-no-long}}
\label{fig:7}
\end{figure}
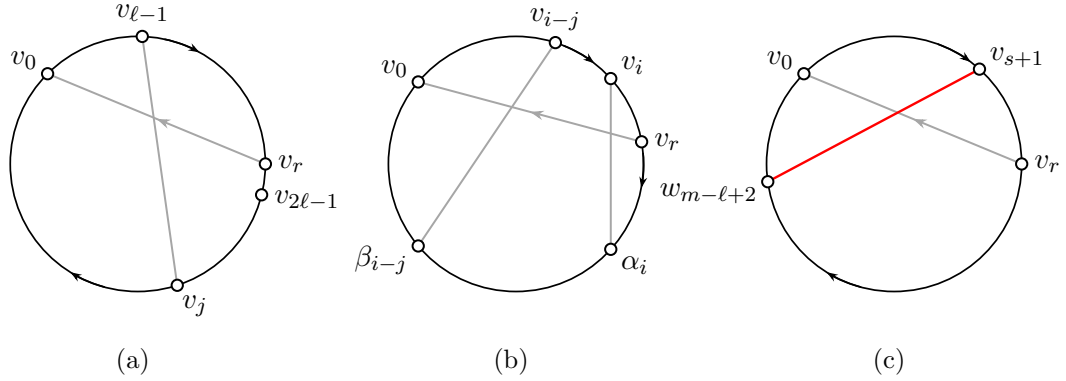

For every $0\leq i\leq n-1$, let $\varepsilon_i=1$ if $F_i\in E(G)$ and  $\varepsilon_i=0$ if $F_i\notin E(G)$. 

\medskip
\textbf{(2.3)} For every $\ell\le t\le r-1$, $N[v_t]\subseteq V(v_{t-\ell}\rbjt{C}v_{t+\ell})$. Moreover $\varepsilon_0= \varepsilon_1=\cdots= \varepsilon_{s-1}=0$ and  $2\le s\le \ell-2$.

If $v_jv_t\in E(G)$ with $0\le j\le t-\ell-1$ or $t+\ell+1\le j\le r+\ell-1$, then $\ell+1\le |t-j|\le r-1$, contradicting the choice of $v_0v_r$.  If $j\ge r+\ell$, $v_0\rbjt{C}v_t$ and $v_r\rbjt{C}v_j$ both have length at least $\ell$.
But by Lemma \ref{lem:rectangle}, the existences of $v_0v_r$ and $v_tv_j$ yield a member of $C(\ell,\ell)$, a contradiction. Thus $N[v_t]\subseteq V(v_{t-\ell}\rbjt{C}v_{t+\ell})$, where $|V(v_{t-\ell}\rbjt{C}v_{t+\ell})|=2\ell+1$. 

Since  $|N[v_t]|\ge 2\ell$,  at most one vertex of $V(v_{t-\ell}\rbjt{C}v_{t+\ell})$ is non-adjacent to $v_t$. Therefore, for $\ell\le t\le r-1$, at least one of $F_{t-\ell}$ and $F_t$ belongs to $E(G)$. Since $r=\ell+s$,  for $0\le j\le s-1$, at least one of $F_{j}$ and $F_{j+\ell}$ belongs to $E(G)$.   Moreover, for any $1\leq j+1\leq s-1$,  if $F_{j+1}\in E(G)$, then by Lemma \ref{lem:antil}, $F_{\ell+j}\notin E(G)$ forcing $F_j\in E(G)$. Consequently,  $\varepsilon_0\ge \varepsilon_1\ge\cdots\ge \varepsilon_{s-1}$.

By (2.2), $F_{\ell-1}\in E(G)$. Then $F_0\notin E(G)$ by Lemma \ref{lem:antil}, i.e., $\varepsilon_0=0$. Hence $\varepsilon_0=\varepsilon_1=\cdots=\varepsilon_{s-1}=0$. If $s\ge \ell-1$,  $F_{\ell-2}\notin E(G)$. Since $0\leq \ell-2\le s-1$,  at least one of $F_{\ell-2}$ and $F_{2\ell-2}$ belongs to $E(G)$. Then $F_{2\ell-2}\in E(G)$. By Lemma~\ref{lem:antil},  $F_{\ell-1}\notin E(G)$, a contradiction. Thus $2\le s\le \ell-2$.

\medskip
Now we aim to deduce a contradiction by using (2.1)-(2.3).
%\textbf{(2.4)}  $(v_r,v_0)\in A(D)$ and $2\le s\le\ell-2$ yield a contradiction.
By (2.1),   $Z=v_0v_1\cdots v_rv_0$ is a directed cycle of length $r+1=\ell+s+1$.
Let $W=w_1\cdots w_m:=v_{r+1}\cdots v_{n-1}$. Note that $r\le n-r$. If $n=2r$, then  $(v_r,v_0)$ is  an $r$-dichord, a contradiction to (2.1). Hence  $n\ge 2r+1\ge 2\ell+2s+1$ and  $m=n-1-r\ge \ell+s=r$.

For every $0\le i\le r$,  let $X_i=\{a\in\{1,\ldots,m\}:v_iw_a\in E(G)\}$. Since $v_i$ has at most $r=\ell+s$ neighbors in $Z$,  $|X_i|\ge(2\ell-1)-(\ell+s)=\ell-s-1 >0$. Let $\alpha_i=\min X_i$ and $\beta_i=\max X_i$.

For $1\le j\le s+1$, $v_i\rbjt{Z}v_{i-j}$ has length $r+1-j\ge\ell$. Hence, if $1\le a<b\le m$, $b-a\ge\ell$, and both $v_iw_a$, $v_{i-j}w_b$ exist, Lemma~\ref{lem:rectangle} forces $b-a\le\ell-1$.
Then for every $1\le j\le s+1$, $\beta_{i-j}\le\alpha_i+\ell-1$ (see in Figure \ref{fig:8_2}). Since $|X_i|\ge \ell-s-1$,  $\alpha_i\le \beta_i-\ell+s+2$. Then we get
\[
\beta_{i-j}\le\beta_i+s+1,\qquad \alpha_i\ge\alpha_{i-j}-(s+1).
\]

By (2.2), $N[v_{\ell-1}]=V(v_0\rbjt{C}v_{2\ell-1})$. So $X_{\ell-1}=\{1,\ldots,\ell-s-1\}$.

We claim that $X_{s+1}=\{m-\ell+s+2,\ldots,m\}$. Note that $s+1={r-\ell+1}$.  For $a\le m-\ell+1$, $v_{s+1}\rbjt{C}v_r\odot(v_r,v_0)$ and $w_a\rbjt{C}v_0$ both have length at least $\ell$. By Lemma~\ref{lem:common-end}, $v_{s+1}w_a\notin E(G)$. When $a=m-\ell+2$ and $v_{s+1}w_a\in E(G)$, $v_0,v_{s+1},v_r,w_a$ appear on $C$ in this order (see in Figure \ref{fig:8_3}), and $|v_{s+1}\rbjt{C}v_r|=\ell-1$ and $|w_a\rbjt{C}v_0|=\ell-1$. By  Lemma~\ref{lem:boundary-crossing} and $(v_r,v_0)\in A(D)$, $D$ is  $C(\ell,\ell)$-free  only if $(w_a,v_{s+1})\in A(D)$. But  $|w_a\rbjt{C}v_{s+1}|=|w_a\rbjt{C}v_0|+|v_0\rbjt{C}v_{s+1}|=r$, which is an $r$-dichord, contradicting (2.1).
For $m-\ell+3\le a\le m-\ell+s+1$,   $\ell+1 \le |w_a\rbjt{C}v_{s+1}| \le r-1$, contradicting the choice of $v_0v_r$.  Hence $X_{s+1}\subseteq\{m-\ell+s+2,\ldots,m\}$, and equality follows from $|X_{s+1}|\ge \ell-s-1$.

Since $\beta_{i-j}\le\beta_i+s+1$ and $\beta_{\ell-1}=\ell-s-1$,
$\beta_{\ell-1-u}\le \ell-s-1+(s+1)\lceil u/(s+1)\rceil$ for $0\le u\le \ell-2$. 

Let $u=\ell-s-2$. Since $\beta_{s+1}=m$, we obtain $m-\ell+s+1\le (s+1)\lceil(\ell-s-2)/(s+1)\rceil$. Also, since $m\ge\ell+s$, $m-\ell+s+1\ge2s+1$. Because $|V(G)|=m+r+1=m+\ell+s+1$ and $N[w_a]\ge 2\ell$, 
every $w_a$ has at most $m-\ell+s+1$ non-neighbors in $G$.

We first determine some non-neighbors of $w_1$ and $w_m$.
If $\ell\le t\le r-1$ and $v_tw_m\in E(G)$, note that $v_0\rbjt{C}v_t$ and $v_r\rbjt{C}w_m$ both have length at least $\ell$, then $v_0v_r$ and $v_tw_m$ yield a member of $C(\ell,\ell)$ by Lemma~\ref{lem:rectangle}, a contradiction. Hence $v_tw_m\notin E(G)$ for every $\ell\le t\le r-1$. If $1\le t\le s$ and $v_tw_1\in E(G)$, then  $v_t\rbjt{C}v_r$ and $w_1\rbjt{C}v_0$ both have lengths at least $\ell$,  then $v_0v_r$ and $v_tw_1$ yield a member of $C(\ell,\ell)$ by Lemma~\ref{lem:rectangle}, a contradiction. Hence $v_tw_1\notin E(G)$ for every $1\le t\le s$.

We claim that $v_{r}w_m \in E(G)$. Write $q=\lfloor (m-\ell+s)/(s+1)\rfloor$. Then $q(s+1)<m-\ell+s+1\le(q+1)(s+1)$. 
Since $q(s+1) < m-\ell+s+1 \le (s+1)\lceil(\ell-s-2)/(s+1)\rceil < \ell-1$, we have $q(s+1) < \ell-2$. For every $0\le u\le q(s+1) < \ell-2$, we have $\beta_{\ell-1-u}\le \ell-s-1+(s+1)\lceil u/(s+1)\rceil\le \ell-s-1+q(s+1)<m$, so $w_m$ is nonadjacent to $q(s+1)+1$ consecutive vertices $v_{\ell-1-q(s+1)},\ldots,v_{\ell-1}$. Since $v_\ell w_m,\ldots,v_{r-1}w_m\notin E(G)$, we have $(q+1)(s+1)$ non-neighbors of $w_m$ in $G$. Since $w_m$ has at most $m-\ell+s+1$ non-neighbors in $G$, $(q+1)(s+1)\le m-\ell+s+1$ and hence $m-\ell+s+1=(q+1)(s+1)$. Therefore, $\{v_{\ell-1-q(s+1)},\ldots,v_{r-1}\}$ are exactly the non-neighbors of $w_m$ in $G$. In particular, $v_rw_m\in E(G)$.

We claim that $v_{0}w_1 \in E(G)$. Since $\alpha_i\ge\alpha_{i-j}-(s+1)$ and $\alpha_{s+1}=m-\ell+s+2$,  $\alpha_{s+1+u}\ge m-\ell+s+2-(s+1)\lceil u/(s+1)\rceil$ for $0\le u\le q(s+1) <  \ell-2$. Since $m-\ell+s+1=q(s+1)+s+1$, it follows that $\alpha_{s+1+u}\ge s+2>1$ for every $0\le u\le q(s+1)$. Hence $w_1$ is nonadjacent to the $q(s+1)+1$ consecutive vertices $v_{s+1},\ldots,v_{s+1+q(s+1)}$. Since $v_1w_1,\ldots,v_sw_1\notin E(G)$ and $q(s+1)+1+s=m-\ell+s+1$, the non-neighbors of $w_1$ in $G$ are exactly $\{v_1,\ldots,v_{s+1+q(s+1)}\}$. In particular, $v_0w_1\in E(G)$.

Since $w_1\rbjt{C}w_m$ and $v_0\rbjt{C}v_r$ both have length at least $\ell$, by Lemma~\ref{lem:rectangle}, and the two edges $v_0w_1$ and $v_rw_m$ yield a contradiction, completing the proof.
\end{proof}

\begin{proof}[\bfseries{Proof of Theorem~\ref{thm:Ll}}]
Assume $D$ is $C(\ell,\ell)$-free. 
When $n=2\ell+1$, since $|N[v_i]| \ge 2\ell=n-1$ for every $0\leq i\le n-1$, each vertex misses at most one local neighbor, so the missing local edges form a matching. When $n \geq 2\ell+2$, by Lemma~\ref{lem:bal-no-long}, $N[v_i] \subseteq V(v_{i-\ell}\rbjt{C}v_{i+\ell})$, where $|V(v_{i-\ell}\rbjt{C}v_{i+\ell})
|=2\ell+1$ for every $0\leq i\leq n-1$. Since $|N[v_i]|  \ge 2\ell$, the missing local edges form a matching.
Then in both cases, the missing local edges form a matching.

By Lemma \ref{lem:antil}, $F_i$ and $F_{i+\ell-1}$ cannot both exist for every $0\leq i\leq n-1$. Suppose $F_i\notin E(G)$ for some $i$. By the above argument, since $F_i\cap F_{i-\ell}=\{v_i\}$ and $F_i,F_{i-\ell}$ are local edges,  $F_{i-\ell}\in E(G)$. By Lemma \ref{lem:antil}, $F_{i-1}\notin E(G)$. 

By repeatedly carrying out the above iteration modulo $n$,  $F_j\notin E(G)$ for every $0\leq j\leq n-1$.
But $F_j\cap F_{j+\ell}=\{v_{j+\ell}\}$, which contradicts the fact that the missing local edges form a matching, completing the proof. 
\end{proof}

\section{Proof of Theorems \ref{thm:main-chi-bound} and \ref{thm:mindeg}}\label{sec-main}

We first give the proof of Theorem \ref{thm:main-chi-bound}
based on Theorem \ref{thm:mindeg}.

\begin{proof}[\bfseries{Proof of Theorem \ref{thm:main-chi-bound}}]
It suffices to show that the underlying graph of   $D$ is $(k+\ell-2)$-degenerate. Suppose the above statement is false, and let $D$ be a counterexample to it with minimum number of vertices. That is, $D$ is a $C(k,\ell)$-free $n$-vertex Hamiltonian digraph   and its underlying graph  $G$ is not $(k+\ell-2)$-degenerate, but the underlying graph of any $C(k,\ell)$-free Hamiltonian digraph  with less than $n$ vertices is $(k +\ell-2)$-degenerate. 

If $\delta(G)\geq k+\ell-1$, then by Theorem \ref{thm:mindeg}, $D$ contains a member of $C(k,\ell)$, a contradiction. Thus  we may assume $d(v_0) \leq k+\ell- 2$. Let $D'$ be the digraph obtained from $D$ by deleting $v_0$ and adding an arc $(v_{n-1}, v_1)$, and $G'$ be the underlying graph of $D'$. Clearly $D'$ is Hamiltonian, $G'$ is an underlying graph of $D'$, and $G - v_0$ is a subgraph of $G'$. Suppose $D'$ contains a member of $C(k, \ell)$, say $H'$. If  $(v_{n-1}, v_1)\notin A(H')$, then $H'$ is a subgraph of $D$, which yields a contradiction. Thus  $(v_{n-1}, v_1)\in A(H')$. However, the digraph obtained from $H'$ by replacing the arc $(v_{n-1}, v_1)$ with the directed path $v_{n-1} v_0v_1$  is a member of $C(k,\ell)$ as well, which is contained in $D$, a contradiction. Hence, the Hamiltonian digraph $D'$ is $C(k,\ell)$-free. Then by our hypothesis, $G'$ is $(k+\ell-2)$-degenerate. Since $G - v_0$ is a subgraph of $G'$ and $d(v_0) \leq k+\ell-2$, it follows that $G$  is also $(k+\ell- 2)$-degenerate. 
\end{proof}

\begin{proof}[\bfseries{Proof of Theorem \ref{thm:mindeg}}]
Since $\delta(G)\geq k+\ell-1$, $n \geq k+\ell$. 
When $n=k+\ell \geq 6$, the condition $\delta(G)\ge k+\ell-1$ forces $G\cong K_{k+\ell}$. Then  $D$ contains a Hamiltonian tournament $T$. When $n=k+\ell \geq 7$ or $(k,\ell) \neq (4,2)$, by Theorem \ref{thm:benhocine},   $D(n,\ell)\subseteq T\subseteq D$, which is a member of $C(k,\ell)$.  When $(k,\ell) = (4,2)$, note that $T_{6}^*$ is not strong, by Theorem \ref{thm:benhocine},  $D(6,2)\subseteq T\subseteq D$, which is a member of  $C(4,2)$.

Suppose $n \geq k+\ell+1 $. If $\ell=1$, the result follows from Theorem~\ref{thm:main-l1}. If $\ell=2$, then $k+\ell\ge6$ implies $k \ge 4$, and Theorem~\ref{thm:L2} gives the desired conclusion. If $\ell\ge3$,  the result follows from Theorems~\ref{thm:Ll+1} and \ref{thm:Ll}. 
\end{proof}

\section{Concluding Remark}

The degeneracy conclusion is specific to  Hamiltonian digraphs, not for all strongly connected digraphs. Indeed, for any positive integers $N$ and $k+\ell\geq 6$, we construct  a   $C(k,\ell)$-free  strongly connected digraph $D$ whose underlying graph has degeneracy $N+1$:
 Take the complete tripartite graph $K_{N,N,1}$ with parts $A,B,\{r\}$, and orient all edges  as $A\to B\to r\to A$. It is easy to check that $D$ is strongly connected and 
its underlying graph has degeneracy $N+1$. Now we prove that $D$ is $C(k,\ell)$-free for any $k+\ell\geq 6$. It is easy to check that  every directed path has length at most $4$, and every directed 
path of length at least $3$ contains $r$ as an internal vertex. Hence there is no member of $C(k,\ell)$ whenever $\max\{k,\ell\}\ge5$ or $\min\{k,\ell\}\ge3$. It remains to consider the case $\{k,\ell\}=\{4,2\}$. Since every directed path $P_1$ of length at least $4$ has endpoints in $A$ and $B$ and contains $r$ as 
an internal vertex, and since every directed path $P_2$ of length at least $2$ with the same endpoints as $P_1$ also 
contains $r$ as an internal vertex, $P_1$ and $P_2$ cannot be internally vertex-disjoint. 

\section*{Acknowledgement}
\noindent This research is supported by National Key R\&D Program of China under grant number 2024YFA1013900 and NSFC under grant number 12471327, 12401447. 
\section*{Declaration}
	
\noindent$\textbf{Conflict~of~interest}$
The authors declare that they have no known competing financial interests or personal relationships that could have appeared to influence the work reported in this paper.
\vskip 2mm	
\noindent$\textbf{Data~availability}$
No data was used for the research described in the article.

\end{document}